\documentclass[11pt]{amsart}
\usepackage[a4paper,margin=0.9in]{geometry}
\usepackage{comment}
\usepackage[colorlinks,citecolor=magenta,linkcolor=black]{hyperref}
\pdfpagewidth=\paperwidth \pdfpageheight=\paperheight
\usepackage{amsfonts,amssymb,amsthm,amsmath,tabu,url}
\usepackage{pgf}
 \usepackage{array}
 \usepackage{tikz-cd}
 \usepackage{pstricks}
 \usepackage{pstricks-add}
 \usepackage{pgf,tikz}
 \usetikzlibrary{automata}
 \usetikzlibrary{arrows}
 \usepackage{indentfirst}
\usepackage{tabularx} 

\DeclareMathOperator{\orb}{orb}
\DeclareMathOperator{\Sing}{Sing}
\DeclareMathOperator{\rk}{rk}

\DeclareMathOperator{\Hom}{Hom}

\newcommand{\SP}{\textrm {SP}}
\newcommand{\ID}{\textrm {ID}}
\def\cC{\mathcal{C}}
\def\cL{\mathcal{L}}

\def\cR{{\mathcal R}}
\def\cA{{\mathcal A}}
\def\cF{\mathcal{F}}
\def\ZZ{\mathbb{Z}}

\def\NN{\mathbb{N}}
\def\CC{\mathbb{C}}

\def\QQ{\mathbb{Q}}

\def\PP{\mathbb{P}}
\def\kk{\Bbbk}
\def\w{\omega}
\def\s{\sigma}
\def\simmaps#1#2{\smash{\mathop{\sim}\limits^{#1}_{#2}}}
\def\surj{\mathbin{\twoheadrightarrow}}
\def\rightmap#1{\smash{\mathop{\rightarrow}\limits^{#1}}}

\theoremstyle{plain}
\newtheorem{thm}{Theorem}
\newtheorem{theorem}[thm]{Theorem}

\newtheorem{lemma}[thm]{Lemma}
\newtheorem{proposition}[thm]{Proposition}
\newtheorem{corollary}[thm]{Corollary}

\theoremstyle{definition}
\newtheorem{definition}[thm]{Definition}
\newtheorem{remark}[thm]{Remark}
\newtheorem{example}[thm]{Example}

\newtheorem{thevarthm}[thm]{\varthmname}

\newenvironment{varthm*}[1]{\trivlist\item[]{\bf #1.}\it}{\endtrivlist}

\def\subclassname{{\bfseries Mathematics Subject Classification
(2020)}\enspace}
\def\subclass#1{\par\addvspace\medskipamount{\rightskip=0pt plus1cm
\def\and{\ifhmode\unskip\nobreak\fi\ $\cdot$
}\noindent\subclassname\ignorespaces#1\par}}

\begin{document}
\title{On modular inequalities for plane projective curves}
\author[Jose Ignacio Cogolludo-Agust\'in]{Jose Ignacio Cogolludo-Agust\'in$^*$}
\thanks{$^*$ Partially supported by a grant of the Spanish Ministerio de Ciencia, Innovación y Universidades MICIU/AEI/10.13039/501100011033/FEDER, UE, reference PID2024-156181NB-C33.}
\address{Dpto. Matemáticas, IUMA, Universidad de Zaragoza, Zaragoza, Spain}
\email{jicogo@unizar.es}
\author[Anca~M\u acinic]{Anca~M\u acinic$^{**}$}
\thanks{$^{**}$ Partially supported by a grant of the Ministry of Research, Innovation and Digitization, CNCS - UEFISCDI, project number PN-IV-P1-PCE-2023-2001, within PNCDI IV}
\address{Simion Stoilow Institute of Mathematics, 
 Bucharest, Romania}
\email{Anca.Macinic@imar.ro}
\date{\today}
\maketitle

\thispagestyle{empty}
\begin{abstract}
We introduce modular inequalities for complements of  plane curves, based on a Combinatorial Aomoto complex construction associated with the weak combinatorial type of a curve. We use this as a tool to investigate twisted Alexander polynomials, in particular to study the characteristic polynomial of the Milnor fiber associated to a projective plane curve, i.e. the classical Alexander polynomial. We give criteria for the non-triviality of the resonance
in degree $1$,
over a positive characteristic field, for curves with quasi fiber--type structure and we compute lower bounds for the multiplicities of some roots of twisted  Alexander polynomials for this type of curves.
We apply these results to theoretically compute  the multiplicities of certain roots of twisted  Alexander polynomials for fiber--type curves.
\end{abstract}
\subclass{14H50, 32S55, 32S40, 
14C21, 14Q05, 14B05}

\section{Introduction}

The Milnor fiber $\cF(\cC)$ of a complex projective curve $\cC\subset \mathbb{P}^{2}$, defined by a homogeneous polynomial $f$ of degree $d$, is the smooth affine hypersurface $\cF(\cC)=\{(x,y,z)\in \mathbb{C}^{3}\setminus \{0\}\mid f(x,y,z)=1\}$. The action of a $d$-th root of unity over $\cF(\cC)$ by $h(x,y,z)=(\xi x,\xi y,\xi z)$ endows the homology $H_1(\cF(\cC);\QQ)$ with a $\QQ[t^{\pm 1}]$--module structure. Since $h_*$ has finite order $d$, this is a semi-simple torsion module. Moreover, since $\QQ[t^{\pm 1}]$ is a PID each of its torsion components is generated by a (Laurent) polynomial associated with $\phi_k$, where $\phi_k(t)$ is a cyclotomic polynomial of order $k|d$. The product of all these polynomials 
\begin{equation}
\label{eq:cAP}
\Delta_\cC(t)=\prod_{k|d} \phi_k^{b_k}(t)
\end{equation}
is referred to as the \emph{Alexander polynomial} of $\cC$. The module $H_1(\cF(\cC);\QQ)$ controls the topology of a $d$-cyclic cover of the complement $\PP^2\setminus \cC$ ramifying with maximal order at every irreducible component of $\cC$, that is, associated with the epimorphism $\varepsilon:H_1(\PP^2\setminus \cC)\to\ZZ_d$ defined as $\varepsilon(\mu_i)=1$ for any meridian $\mu_i$ around an irreducible component $\cC_i$ of $\cC$. If one is interested in cyclic $d$-covers of $\PP^2\setminus \cC$ associated with other epimorphisms $\varepsilon:H_1(\PP^2\setminus \cC)\to\ZZ_d$ defined as $\varepsilon(\mu_i)=m_i\in\ZZ_d$ it is useful to define the \emph{twisted} Milnor fiber $\cF(\cC,\bar m)=\{(x,y,z)\in \mathbb{C}^{3}\setminus \{0\}\mid F_1^{m_1}\cdots F_r^{m_r}(x,y,z)=1\}$ associated with the weights $\bar m=(m_1,\dots,m_r)$ as long as $\gcd(\bar m)=1$, which is the condition for $\varepsilon$ to be an epimorphism.
The torsion of the $\QQ[t^{\pm 1}]$--module $H_1(\cF(\cC,\bar m);\QQ)$
also defines the \emph{twisted Alexander polynomial} 
\begin{equation}
\label{eq:twAP}
\Delta_{(\mathcal{C},\bar m)}(t)=\prod_{k|d} \phi_k^{b_{k,\bar m}}(t).
\end{equation}

For complex projective line arrangements $\cL$, Papadima--Suciu~\cite{PS2} provide {\it combinatorial} upper bounds for 
$b_k$ in~\eqref{eq:cAP} for the Alexander polynomial $\Delta_{\cL}(t)$ and $k$ a prime power. In the theory of arrangements of hyperplanes, the {\it combinatorics} of an arrangement reverberates into the topological and algebraic properties of the arrangement and of its complement. Let us mention, for instance, the emblematic result by Brieskorn~\cite{Brieskorn-tresses} and completed by Orlik--Solomon in~\cite{OS} which states that the cohomology algebra of the complement of an arrangement of hyperplanes with integer coefficients is completely determined by the intersection lattice of the arrangement.

The purpose of this paper is to give estimates for the multiplicities $b_k$ of the classical Alexander polynomial in~\eqref{eq:cAP} and $b_{k,\bar m}$ for the twisted version in~\eqref{eq:twAP} in case $k=p^s$, for $p$ prime.
In order to do that we use two different tools that offer respectively upper and lower bounds to these multiplicities~$b_{k}$ and $b_{k,\bar m}$. 

Upper bounds are obtained via the \emph{weak combinatorial type} of a curve $\cC$, which can be seen as a generalization to curves of the classical concept of combinatorics of line arrangements, developed by the first author in~\cite{C} and jointly with Matei in~\cite{CM}. The authors prove that the weak combinatorial type of a curve determines its cohomology ring. 

The cohomology ring over a finite field $\kk$ produces a complex given by multiplication by a 1-form $\w\in H^1(X;\kk)$ denoted by 
$(H^*(X;\kk),\cdot\wedge\w)$ 
and called the \emph{Aomoto Complex}. Note that the cohomology of these complexes produces a filtration of $H^1(X;\kk)$ called the \emph{resonance varieties} of $X$ and defined as
$$\cR_k(X,\kk):=\{\w \in H^1(X;\kk) \mid \dim_\kk H^1(H^*(X;\kk),\cdot\wedge\w)\geq k\}.$$
On the other hand, given the weak combinatorial type of a curve we define in~\S\ref{sec:gen_Aomoto} a graded algebra $A_{W,\kk}^*$ over a field $\kk$. Multiplication by an element of order one $\w\in A_{W,\kk}^1$ also produces a family of complexes called \emph{Combinatorial Aomoto Complex} and denoted by $(A_{W,\kk}^1,\cdot\curlywedge\w)$
whose cohomology produces a filtration of $A_{W,\kk}^1$ called the \emph{combinatorial resonance varieties} of $W$ and defined as
$$\cR_k(W,\kk):=\{\omega \in A_{W,\kk}^1 \mid \dim_\kk H^1(A_{W,\kk}^*,\cdot\curlywedge \w)\geq k\}.$$
Note that 
$$\{0\}=\dots =\cR_{-1}(W,\kk)=\cR_0(W,\kk)\subseteq \cR_1(W,\kk)\subseteq\dots\subseteq\cR_r(W,\kk)=\cR_{r+1}(W,\kk)=\dots = A^1_{W,\kk},$$
where $r=\dim_\kk A^1_{W,\kk}$.
Both families of complexes are neither isomorphic nor quasi-isomorphic in general; however, one of our main results is Theorem~\ref{thm:resonance} which states that there is a morphism of complexes that allows one to identify their resonance varieties, namely,
$$\cR_k(W,\kk)\cong\cR_k(X,\kk).$$

Combining this with~\cite[Theorem C]{PS2} we compute in \S\ref{sec:generalcase}, Theorem \ref{thm:inequality}, upper bounds on $b_k$ for $k$ a prime power, using combinatorial resonance varieties of the weak combinatorial type of the curve. We illustrate this result by computing bounds for the multiplicities of roots of the classical Alexander polynomials of order powers of  primes, for a series of examples.

We also present in~\S\ref{subsec:triv_crit} a triviality criterion for the cohomology in degree $1$ of the Combinatorial Aomoto complex with coefficients in a finite field. This in turn enables us to formulate combinatorial conditions to ensure $b_k=0$, see Theorem \ref{thm:completelypreductive}.
Section \S\ref{subsec:twisted} is devoted to the twisted counterpart of the previous results. In addition, in 
\S\ref{subsec:milnor} we present some examples.

A tool for obtaining lower bounds for $b_k$, and more generally for $b_{k,\bar m}$, is presented in \S\ref{sec:pencils} and applied to \emph{(quasi) fiber-type curves}, that is, the class of curves that can be obtained as a union of fibers of a pencil. These lower bounds depend on certain properties of the pencil.
This class contains as a subclass the fiber--type curves, also known as pencil--type curves, in particular the Halpen pencils of index $k \geq 1$ defined in \cite{DLPU}, and the projective line arrangements that admit combinatorial pencil/multinet structures as defined in~\cite{CM,FY}. 

The lower bounds for $b_{k,\bar m}$ in the  (quasi) fiber-type curve case are given precisely by the number and multiplicity of the non-reduced fibers of the pencil $\mathcal{C}$ is a union of, but also, and more surprisingly, by the extra multiple fibers that are not part of~$\mathcal{C}$, see Theorem~\ref{thm:roots}. This generalizes several recent results, for instance \cite[Theorems 3.2 -- 3.4]{DPS}. 
It is worth noticing that this result applies not only to roots of unity of order prime powers as the results in section~\S\ref{sec:generalcase} did.

In \S\ref{sec:applications} we present a series of applications to reprove known results by Dimca-Pokora-Sticlaru in~\cite{DPS} on Halphen curves and to provide new examples of calculations of $b_k$ when it coincides with the upper and lower bounds obtained above.

Finally, in \S\ref{sec:Yoshinaga}
we revisit Yoshinaga's counterexample to a conjecture by Papadima--Suciu~\cite{PS3}. The conjecture questioned the existence, in the class of arrangements of hyperplanes, of examples for which the modular inequalities are strict. We approach Yoshinaga's example from a new perspective, i.e. viewed as a quasi fiber--type curve providing an alternative proof.

\section*{Acknowledgments}
We would like to thank to Piotr Pokora for many discussions during the elaboration of this note and for advising  on the Halphen pencils examples from  \S\ref{subsec:milnor}.

\section{Preliminaries}

Let $\mathcal{C} \subset \mathbb{C}\mathbb{P}^2$ be a reduced curve of degree $d$ and let $ \mathcal{C}:=\cup_{i=0}^r \mathcal{C}_i$, where $\mathcal{C}_i:=\{ F_i =0\}$, be the decomposition of $\mathcal{C}$ into irreducible components. 
Denote $\mathbb{P}^2:=\mathbb{C} \mathbb{P}^{2} $.

\subsection{Torsion freeness of the homology of the complement of a projective curve}

An initial condition of torsion freeness on the homology of the space $\mathbb P^2\setminus \mathcal{C}$ is required in our context (\cite[Theorem C]{PS3}). This follows from a well-known result in the context of complements of hypersurfaces in the complex projective space (see for instance~\cite[\S4 Prop. 1.3]{Dimca-singularities}).

\begin{proposition}
\label{prop:complementC_tors_free}
 In the above notation, if  ${\gcd}(\deg(F_0),\dots,\deg(F_r)) = 1$, then $H_1(\mathbb{P}^2 \setminus \mathcal{C}; \mathbb{Z})\cong \mathbb{Z}^{r}$.
\end{proposition}



\subsection{Milnor fiber and cohomology with local coefficients}
\label{ss:milnor&local_coeff}
For a reduced degree $d$ curve $\mathcal{C}$ in $\mathbb{P}^2$ as before, we have a global Milnor fibration $$\mathbb{C}^3 \setminus \mathcal{C}\xrightarrow{F_0 \cdots F_r}\mathbb{C}^*.$$  Consider the associated  {\it Milnor fiber} $\cF(\cC):= \{x\in\mathbb{C}^3\mid (F_0 \cdots F_r)(x)=1\}\subset\mathbb{C}^3$.
There is an action on $\mathbb{C}^{3}$ defined by the multiplication by a primitive root of unity of degree $d$. Since $F_i$ is homogeneous, this restricts to an action on $\cF(\cC)$:
$$h: \cF(\cC) \rightarrow \cF(\cC),$$
 the so-called {\it geometric monodromy}. This in turn induces an action 
$$h_*: H_*(\cF(\cC); \mathbb{Q})  \rightarrow H_*(\cF(\cC); \mathbb{Q})$$
 on the homology groups of the Milnor fiber, the {\it algebraic monodromy}. This action defines a $\mathbb{Q}[t^{\pm 1}]$-module structure on each homology group $H_*(\cF(\cC); \mathbb{Q})$, where $t$ acts as the algebraic monodromy. Since $h^d = id$,  $H_*(\cF(\cC); \mathbb{Q})$ is a torsion module that is annihilated by $(t^d-1) = \Pi_{k | d} \phi_k$, where $\phi_k$ is the {\it k}-th cyclotomic  polynomial.

Computing the algebraic monodromy is a difficult and open problem, even for $*=1$, see e.g. \cite{MP} or \cite{MPP} for results on the algebraic monodromy of the first rational homology module of the Milnor fiber for graphic arrangements or for complex reflection arrangements.

Let us recall the decomposition of $H_1(\cF(\cC); \mathbb{Q})$ with respect to the algebraic monodromy action:
$$
H_1(\cF(\cC);\mathbb{Q}) = \bigoplus_{k | d} (\mathbb{Q}[t^{\pm 1}]/ \phi_k)^{b_k}.
$$
We define the Alexander polynomial of the curve $\mathcal{C}$, denoted by $\Delta_{\mathcal{C}}(t)$, by the formula

\begin{equation}
\label{eq:alexander_poly}
\Delta_{\mathcal{C}}(t) = \prod_{k | d} \phi_{k}^{b_k}.
\end{equation}
The exponents $b_k$ in the above formula are connected to the cohomology of the complement $ \tilde X:= \mathbb{P}^2 \setminus \cC$ with coefficients in certain rank-one local systems. Let us briefly recall this connection.
Denote by $H^1(\tilde X; \mathbb{C}_{\rho})$ the cohomology with coefficients in the rank-one local system $\rho: \pi_1(\tilde X) \rightarrow \mathbb{C}^*$.
If $\rho(x) = \xi, \; \xi \in \CC^*$ fixed, for all elements $x$ of the $\mathbb{Z}$-basis, i.e.,  if $\rho$ is constant on the $\mathbb{Z}$-basis,  then the corresponding local system is called {\it equimonodromical}. In this case, we denote the local system $\mathbb{C}_{\rho}$ as $\mathbb{C}_{\xi}$ by convention.
For $k \in \mathbb{Z}_{>1}$, let $\xi_k \in \mathbb{C}$ be a primitive $k$-th root of unity.
Then it is known that
the multiplicity of $\xi_k$ as a root of the Alexander polynomial, i.e. the exponents $b_k$ from \eqref{eq:alexander_poly}, satisfy the following equality
\begin{equation}
\label{eq:local_betti}
b_k = {\rm dim}_{\mathbb{C}}H^1(\tilde X;\mathbb{C}_{\xi_k}).
\end{equation}

For a generalization of the above constructions to non-reduced curves, see \S\ref{subsec:twisted}.

\subsection {Aomoto complexes}
We need one more ingredient before we can introduce modular inequalities.
Let $X$ be a finite type CW-complex and let $\kk$ be an arbitrary field. If $\omega \in H^1(X;\kk)$ is such that $\omega^2=0$, one can define a cochain complex structure on the cohomology ring $H^*(X;\kk)$ with the differential defined by the multiplication by $\omega$:

\begin{equation}
\label{eq:Aomoto_complex}
H^0(X;\kk) \overset{ \cdot \wedge \omega} \longrightarrow H^1 (X;\kk) \overset{ \cdot\wedge \omega} \longrightarrow H^2(X;\kk) \overset{\cdot \wedge \omega} \longrightarrow \dots
\end{equation}
This cochain complex is called the {\it Aomoto complex associated with~$\omega$}. 
\subsection{Modular inequalities for curve complements}
Generalizing a result by Cogolludo for rational curves from \cite{C}, Cogolludo-Matei give in \cite{CM} a presentation of the cohomology module of  the complement of an arbitrary reduced curve in the complex projective plane. Consider a projective curve $\cC\subset\PP^2$ and let $\cC_0$ be a line transversal to $\cC$. 
We define $X=\PP^2\setminus (\cC_0\cup\cC)$ and we will refer to $X$ as the \emph{affine complement} of $\cC$. Note that the topology of $X$ does not depend on the choice of $\cC_0$.
It follows by \cite[Theorem 3.1]{CM} that $H^1(X;\CC)$ is generated, as a vector space, 
by the logarithmic $1$-forms $\sigma_i:=d\log\left({F_i} / {F_0^{d_i}}\right)$ with $1 \leq i \leq r$, where $F_i$ is an irreducible polynomial defining the irreducible component $\cC_i$ of $\mathcal{C}$, as in \cite[Section 3]{CM}. 

Assume there is a preferred set of generators for $\pi_1(X)$, then a homomorphism $\omega : \pi_1(X) \rightarrow \mathbb{Z}$ is completely  defined by a set of $r$ integers $(m_i)_{1 \leq i \leq r}$, the images of the generators of $ \pi_1(X)$ in $\mathbb{Z}$. Note that $\Hom_{\mathbb{Z}}(\pi_1(X), \mathbb{Z}) \equiv \Hom_{\mathbb{Z}}(H_1(X), \mathbb{Z})  \equiv H^1(X;\ZZ)$, so we can identify $\omega$ with an element  $\w=\sum_i m_i \sigma_i$ in $H^1(X)$.

For a given complex number $\zeta \in \mathbb{C}^*$, one can define a rank-one local system $\rho_{\omega, \zeta}$ of coefficients, i.e., an element of $\Hom_{\mathbb{Z}}(\pi_1(X), \mathbb{C}^*)$ by 
$$\rho_{\omega, \zeta}(x):=\zeta^{\omega(x)}.$$

If $\zeta$ is a primitive root of unity, then $\rho_{\omega, \zeta}$ is called a {\it rational} local system.\\

By Proposition \ref{prop:complementC_tors_free}, the integral homology of the affine complement $X$ of $\cC$ is torsion free. 
It follows that the modular inequalities stated by Papadima-Suciu in \cite[Theorem C]{PS2} hold for $X$. 
We recall a version of their result stated in our particular context. 
In order to do so, let us fix some notation.
Let $\kk=\ZZ_p$ for a prime $p$ denote the finite field with $p$ elements and let $\pi_1(X) \overset{\omega }{\longrightarrow} \mathbb{Z}\to \kk$ denote the homomorphism defined by the classes of the integers $(m_i)_{1 \leq i \leq r}$ in $\kk$. Denote 
$\omega_{\kk} = \sum_i m_i \sigma_i \in H^1(X;\kk)$.

\begin{theorem}{{\rm{\cite[Theorem C]{PS2}}}}
\label{thm:thmC_PS}
Let $X$ be the affine complement of a projective curve $\cC$, $p$ a prime number, $\xi$ a primitive root of the unity of order $p^s$, and let $\rho =  \rho_{\omega, \xi}$. Then 
\begin{equation}
\label{eq:mod_PS_gen} 
\dim_{\mathbb{C}}H^1(X;\mathbb{C}_\rho)  \leq {\rm dim}_{\kk} H^1(H^*(X;\kk), \cdot\wedge\omega_{\kk}).
\end{equation}
\end{theorem}

\begin{remark}
\label{rem:proj/affine}
If $X$ is the affine complement of the curve $\cC$ and $\tilde X:=\PP^2 \setminus \cC$, then we have a surjection
$$\pi_1(X) \overset{i_*}\surj \pi_1(\tilde X)$$ induced by the inclusion $X \hookrightarrow \tilde X$. Consequently, any character $\rho: \pi_1(\tilde X) \rightarrow \CC^*$ induces a character $i_*\rho: \pi_1(X) \rightarrow \CC^*$. In particular, 
$$ {\rm dim}_{\mathbb{C}}H^1(\tilde X;\mathbb{C}_{\rho}) \leq  {\rm dim}_{\mathbb{C}}H^1(X;\mathbb{C}_{i_*\rho}).$$ 
\end{remark}

By \eqref{eq:local_betti} and Remark  \ref{rem:proj/affine}, Theorem \ref{thm:thmC_PS} provides an upper bound on the multiplicities $b_k$ of the primitive roots of unity of order $k$ of the Alexander polynomial $\Delta_{\cC}$, when $k=p^s$ for some prime $p$ and some integer $s \geq 1$:

 \begin{equation}
 \label{eq:mod_ineq}
b_k \leq {\rm dim}_{\kk} H^1(H^*(X;\kk), \cdot \wedge \omega_1),
 \end{equation}
 where $\omega_1:= \sum_i \sigma_i \in H^1(X;{\kk})$.

\section{The Combinatorial Aomoto Complex}
\label{sec:gen_Aomoto}
The purpose of this section is to present a complex that is quasi-isomorphic to the standard Aomoto complex,
and whose definition only depends on the weak combinatorial type of a curve.

\subsection{The weak combinatorial type of a projective curve}
\label{ss:weak_comb}
We will assume that $\cC=\cC_1\cup \ldots \cup\cC_r$ is a plane projective curve, that is, the reduced structure of the zero set of a homogeneous polynomial in three variables. Now we define the
\emph{weak combinatorial type of $\cC$} as follows.
Let $S=\Sing(\cC)$ be the collection of singular points of $\cC$ and $\mathbf r=\{1,...,r\}$ the set of
subindices of the irreducible components of $\cC$.
For any $P\in S$, let $\Delta_P$ be the collection of irreducible local branches of $\cC$ at $P$.
Consider the following map $\phi_P:\Delta_P\to \mathbf r$, where $\phi_P(\delta)=i$ if $\cC_i$ is such that $\delta$ is a branch in $(\cC_i,P)$. Define by $\SP^2_{\phi_P}(\Delta_P)$ the symmetric
sum of $\Delta_P$ outside of the $\phi_P$-diagonal, that~is
$$
\SP^2_{\phi_P}(\Delta_P):=\{(\delta_1,\delta_2)\in \Delta_P^2\mid \phi_P(\delta_1)\neq\phi_P(\delta_2)\}/\Sigma_2.
$$
One has the map $\mu_P:\SP^2_{\phi_P}(\Delta_P)\to \NN$, where $\mu_P(\delta_1,\delta_2)$ is the intersection
number of $\delta_1$ and $\delta_2$ at~$P$.
Then the list $W_\cC:=(\mathbf r,S,\{\Delta_P,\phi_P,\mu_P\}_{P\in S})$ is the \emph{weak combinatorial type} of $\cC$.
In~\cite{CM}, the \emph{combinatorial type} of $\cC$ is defined as $\tilde W_\cC:=(W_\cC,\bar d,\bar g)$, where $W_\cC$ is the weak combinatorial type of $\cC$, $\bar d = (d_1,\dots,d_r)$ is the list of degrees, and $\bar g$ is the list of genera of the irreducible components of $\cC$. Note that, according to~\cite{CM}, $\tilde W_\cC$ determines the ring structure $H^*(X;\ZZ)$ of the affine complement $X$ of~$\cC$ with respect to a line at infinity in general position.

\subsection{A graded algebra associated with a weak combinatorial type}
\label{subsect:GAC}
Let $\cC$ be a plane curve and $W=(\mathbf r,S,\{\Delta_P,\phi_P,\mu_P\}_{P\in S})$ its weak combinatorial type,
$\kk=\ZZ_p$ with $p$ being prime. 
For simplicity we will assume $\# \mathbf r>2$. In that case one can recover the degrees $d_i$ of the components as follows. For every pair of elements $i,j\in \mathbf r$, $i\neq j$, define $$d_{ij}:=\sum_{\substack{P \in S,\\  \phi_P(\delta_i)=i,\\ 
\phi_P(\delta_j)=j}}
\mu_P(\delta_i,\delta_j).$$
Note that $d_{ij}=d_{ji}=\cC_i\cdot \cC_j$. Then one can recover $d_i$ by considering any triple of different indices $i,j,k\in \mathbf r$ (which exists by hypothesis) as
$
d_i:=\sqrt{\frac{d_{ij}d_{ik}}{d_{jk}}}.
$

We consider
\begin{enumerate}
\item $A_W^0:=\kk$,
\item
$A_W^1:=\bigoplus_{i\in \mathbf r} \sigma_i\kk$,
\item
$A_W^2:=A_W^{2,0}\oplus A_W^{2,\infty}$, where
$$
A_W^{2,\infty}:=\langle \bar\psi^i_\infty : i\in\mathbf{r}\rangle_\kk/\langle \bar\psi^i_\infty: p\nmid d_i \rangle_\kk,\quad
A_W^{2,0}:=\bigoplus_{P\in S} \frac{A_P\wedge A_P}{I_P},\quad
A_P:=\langle \psi_P^\delta : \delta\in \Delta_P\rangle_\kk
$$
and
\begin{equation}
\label{eq-ip}
I_P:=\{
\psi_P^{\delta_1}\wedge \psi_P^{\delta_2}+\psi_P^{\delta_2}\wedge \psi_P^{\delta_3}+
\psi_P^{\delta_3}\wedge \psi_P^{\delta_1}\}+
\{\psi_P^{\delta_1}\wedge \psi_P^{\delta_2}\mid \#\phi_P(\Delta_P)=1\}\subset A_P\wedge A_P.
\end{equation}
\item $A_W^i=0$ ($i\geq 3$).
\end{enumerate}
Finally, we turn $A^*_W$ into a graded algebra by means of the following product
\begin{equation}
\label{eq-prod}
\sigma_i\wedge \sigma_j:=
\sum_{\substack{P \in S,\\  \delta_1\in\phi_P^{-1}(i),\\ \delta_2\in\phi_P^{-1}(j)}}
\mu_P(\delta_1,\delta_2)\ \psi_P^{\delta_1}\wedge \psi_P^{\delta_2}
+d_j\bar\psi_\infty^i-d_i\bar\psi_\infty^j.
\end{equation}

\begin{definition}
\label{def:GAComplex}
Given an element of degree one, say $\omega\in A_{W,\kk}^1$, one can consider the complex
$(A_{W,\kk}^*,\cdot\wedge\w)$ described by using the algebra structure as follows:
\begin{equation}
\label{eq:GAComplex}
\array{ccccccccc}
0 & \to & A_W^0 & \rightmap{\wedge \omega} & A_W^1 & \rightmap{\wedge \omega} & A_W^2 & \to & 0\\
&& k & \mapsto & k\omega &&&&\\
&&&& \alpha & \mapsto & \alpha \wedge \omega. &&
\endarray
\end{equation}
Such a complex $(A^*_{W,\kk},\cdot\wedge\w)$ will be called the \emph{Combinatorial Aomoto Complex of $W$} with coefficients in~$\kk$. The reference to $\kk$ will be omitted if the field is clear from the context.
\end{definition}

\begin{remark}
\label{rem-basis-A2}
In order to describe a matrix of the homomorphism
$A_W^1 \ \rightmap{\wedge \omega} \ A_W^2$ we need a basis for
$A_W^2$ (note that $\{\sigma_i\}_{i\in \mathbf r}$ is a basis for the vector space $A_W^1$).
We will present a basis that will come in handy in the future.
\begin{itemize}
\label{rem:basis-A2}
\item
For any point $P\in S$ one can choose a local branch $\bar\delta\in \Delta_P$
(resp. an element $\psi_P^{\bar\delta}\in A_P$), which will be referred to as the
\emph{preferred branch at $P$} (resp. \emph{preferred form at $P$}).
\item
Using this notation, we observe that the set
\begin{equation}
\label{eq:basis-A2}
\{\bar\psi_P^{\delta}:=\psi_P^{\bar\delta}\wedge \psi_P^{\delta}\}_{P\in S}
\end{equation}
is a basis of the vector space~$A^{2,0}_W$.
\end{itemize}
One has the relation 
$\psi_P^{\delta_i}\wedge\psi_P^{\delta_j}=\bar\psi_P^{\delta_j}-\bar\psi_P^{\delta_i},$
where $\bar\psi_P^{\bar \delta}=0$ for the preferred branch $\bar\delta$ by convention.

With respect to this basis, the product $A^1_W\wedge A^1_W$ can be defined as
\begin{equation}
\label{eq:product2}
\sigma_i\wedge\sigma_j:=
\sum_{P\in S} (\mu_P(\delta_j,\cC_i)\bar\psi_P^{\delta_j}-\mu_P(\delta_i,\cC_j)\bar\psi_P^{\delta_i})
+d_j\bar\psi^i_\infty-d_i\bar\psi^j_\infty,
\end{equation}
extending by bi-linearity, where $\mu_P(\delta_i,\cC_j)$ denotes the multiplicity of the intersection of
the branch $\delta_i$ and the curve $\cC_j$ at $P$, i.e.,
${\displaystyle{
\mu_P(\delta_i,\cC_j)=\sum_{\phi_P(\delta_j)=j} \mu_P(\delta_i,\delta_j)}}$.
\end{remark}

As usual, one can consider the \emph{cohomology jumping loci} of such a family of complexes
parametrized by $A_W^1$ to obtain the family of resonance varieties associated with $A_{W,\kk}^*$.

\begin{definition}
The subset of $A_{W,\kk}^1$
$$\cR_k(W,\kk):=\{\omega \in A_{W,\kk}^1 \mid \dim_\kk H^1(A_{W,\kk}^*,\cdot\wedge \w)\geq k\}$$
is called the \emph{$k$-th combinatorial resonance variety} of $W$. 
\end{definition}

Associated with the affine complement of $\cC$, i.e., $X=\PP^2\setminus (\cC_0\cup\cC)=\CC^2\setminus\cC$ and with
any $1$-cohomology form $\omega\in H^1(X;\kk)$, one can define the cohomology Aomoto complex with respect to
a field $\kk$ as $(H^*(X;\kk),\cdot\wedge\w)$. The \emph{$k$-th resonance variety}
$\cR_k(\cC,\kk)$ of $\cC$ can be defined as the jumping loci of the family of cohomology Aomoto complexes
parametrized by
$$
\cR_k(\cC,\kk)=\{\w\in H^1(X;\kk)\mid \dim_\kk H^1(H^*(X;\kk),\cdot\wedge\w)\geq k\}
$$

\begin{thm}
\label{thm:resonance}
Under the above hypotheses, $\cR_k(\cC,\kk)\cong\cR_k(W_\cC,\kk)$.
\end{thm}

\begin{proof}
Since the $k$-th resonance variety $\cR_k(\cC,\kk)$ can be computed as the first cohomology jumping loci of the Aomoto complex $(H^*(X;\kk),\cdot\wedge\w)$ of the affine complement $X$ of $\cC$, it is enough to show that the Aomoto complex $(H^*(X;\kk),\cdot\wedge\w)$ is 1-quasi-isomorphic to the Combinatorial Aomoto complex $(A^*_{W_\cC},\cdot\wedge\w)$, that is,
there is a morphism of complexes producing $H^k(A^*_{W_\cC},\cdot\wedge\w)\cong H^k(H^*(X;\kk),\cdot\wedge\w)$ for $k\leq 1$. 
In order to distinguish both products, we will denote by $\cdot\curlywedge\w$ the product in $A^*_{W_{\cC}}$ as defined in~\eqref{eq:product2} and by $\cdot\wedge\w$ the product in~$H^*(X;\kk)$.

For simplicity, let us denote $W=W_\cC$.
Note that $H^k(X;\kk)\cong A^k_W$ as long as $k\neq 2$. Recall from \cite{CM} that $H^2(X;\kk)$ can be decomposed as
$H^2(X;\kk)=V_0\oplus V_\infty\oplus V_g\oplus \bar V_g$, where $V_0\cong A^{2,0}_W$. 
This isomorphism is given by $\psi_P^{\delta,\delta'}\mapsto \psi_P^\delta\wedge\psi_P^{\delta'}$.
One can consider an intermediate graded ring $\tilde A^*$ defined as
$$
\tilde A^k=
\begin{cases}
V_0\oplus V_\infty\cong H^2(X;\kk)/(V_g\oplus \bar V_g) & \textrm{ if } k = 2\\
A^k_W & \textrm{ otherwise,}
\end{cases}
$$
where $\sigma_i\wedge\sigma_j$ is inherited by the product in $H^*(X;\kk)$. Note that
there is a projection morphism $H^k(X;\kk)\to \tilde A^k$ satisfying that
$\sigma_i\wedge\sigma_j\in V_0\oplus V_\infty$, which implies 
$H^1(H^*(X;\kk),\cdot\wedge\w)\cong H^1(\tilde A^*,\cdot\wedge\w)$,  
and hence
the projection morphism $(H^*(X;\kk),\cdot\wedge\w)\to (\tilde A^*,\cdot\wedge\w)$ is a 1-quasi-isomorphism.

Consider now the morphism $(A^*_W,\cdot\curlywedge\w) \to  (\tilde A^*,\cdot\wedge\w)$
given by the identity if $* \neq 2$ and, for $*=2$, defined on the generators by
$\psi_P^\delta\wedge\psi_P^{\delta'} \mapsto \psi_P^{\delta,\delta'}$, 
$\bar \psi_\infty^i\mapsto \sum_{k_i=2}^{d_i}\psi_\infty^{i,k_i}$. 
Note that we can consider $A^2_W \subset \tilde A^2$ via this correspondence.
We want to show that
 $(A^*_W,\cdot\curlywedge\w) \to  (\tilde A^*,\cdot\wedge\w)$ is a 1-quasi-isomorphism.
 This is equivalent to proving that
$\tilde K=\ker \tilde A^1\xrightarrow{\wedge\omega}\tilde A^2=\ker A_W^1\xrightarrow{\curlywedge\omega}A_W^2=K$.
By construction, it is clear that $\tilde K\subset K$. For the converse, assume
$\alpha=\sum_{i} \alpha_i\sigma_i\in \tilde A^1=A^1_W$,
$\omega=\sum_i \beta_i\sigma_i\in \tilde A^1=A^1_W$. Note that
\begin{equation}
\label{eq:awb}
\alpha\wedge \omega=\sum_P \lambda_{\delta,P} \bar\psi_P^\delta + \sum_{i=1}^r \lambda_i \bar\psi^i_\infty.
\end{equation}
where the equality \eqref{eq:awb} is seen in $\tilde A^2$,  $\bar\psi^i_\infty:=\sum_{k_i=2}^{d_i}\psi^{i,k_i}_\infty$ and $\bar\psi_P^\delta := \psi_P^{\bar \delta, \delta}$.
One needs to show that $\lambda_i=0$ if $\alpha \in K$ and $p\nmid d_i$.
Using definition~\eqref{eq:product2}, one obtains
\begin{equation}
\label{eq:lambda_i}
    \lambda_i=\sum_{k=1}^r d_k(\alpha_k\beta_i-\alpha_i\beta_k)=
\left| \array{cc} \beta_i & \alpha_i \\ \sum_k d_k\beta_k & \sum_k d_k\alpha_k \endarray\right|.
\end{equation}
On the other hand, by~\eqref{eq:awb} and~\eqref{eq:product2}, if $\phi_P(\delta)=i$, then one has
$$
\lambda_{\delta,P}=\sum_{k=1}^r \mu_P(\delta,\cC_k)(\alpha_k\beta_i-\alpha_i\beta_k).
$$
The condition $\alpha \in K$ implies $\lambda_{\delta,P}=0$, i.e.
\begin{equation}
\label{eq:det}
\left| \array{cc} \beta_i & \alpha_i \\
\sum_k \mu_P(\delta,\cC_k)\beta_k & \sum_k \mu_P(\delta,\cC_k)\alpha_k \endarray \right|=0.
\end{equation}
Note that, Bezout's theorem implies $\sum_{P\in S,\phi_P(\delta)=i} \sum_k \mu_P(\delta,\cC_k)\beta_k = 
d_i \sum_k d_k \beta_k$ and, likewise, $\sum_{P\in S,\phi_P(\delta)=i} \sum_k \mu_P(\delta,\cC_k)\alpha_k =d_i \sum_k d_k \alpha_k$.
Hence, summing the determinants \eqref{eq:det} over $(P\in S,\phi_P(\delta)=i)$ we get
 $$0=\sum_{P\in S,\phi_P(\delta)=i}\lambda_{\delta,P}= \left| \array{cc} \beta_i & \alpha_i \\
 d_i \sum_k d_k \beta_k & d_i \sum_k d_k \alpha_k \endarray \right| = d_i\lambda_i,$$ where the last equality follows by \eqref{eq:lambda_i}. Hence $d_i\lambda_i=0$, which implies $\lambda_i=0$, since we are assuming $p\nmid d_i$.
This ends the proof.
\end{proof}

The conditions described in~\eqref{eq:det} will be frequently used in the rest of the paper. We will often refer to them as \emph{resonance conditions}.

\begin{corollary}
\label{cor:eq_det}
Let $\alpha, \beta \in A^1_{W_\cC}$ given by $\alpha = \sum_i \alpha_i \sigma_i, \; \beta = \sum_i \beta_i \sigma_i$. Then $\alpha \curlywedge \beta=0$ implies that the resonance conditions \eqref{eq:det} hold for all $P \in \Sing(\cC)$ and all $\delta \in \Delta_P$. Moreover, the resonance conditions~\eqref{eq:det} are also sufficient in the statement above if $p\nmid d_i$ for all $i\in\mathbf{r}$.
\end{corollary}

\begin{proof}
It follows immediately from the proof of Theorem \ref{thm:resonance}.
\end{proof}

\begin{remark}
\label{rem:Falk_matroid_thms}
Corollary~\ref{cor:eq_det} is an analog of \cite[Theorem 3.5]{F}, a result on matroids, and implicitly of \cite[Theorem 2.5]{F2}, its version for realizable matroids, i.e. for hyperplane arrangements. The construction of the Combinatorial Aomoto complex described in \ref{subsect:GAC} also works for an Abstract Curve Combinatorics as defined in \cite{BC}. In this context, therefore, an Abstract Curve Combinatorics generalizes the notion of a rank-three matroid.
\end{remark}

\begin{corollary}
\label{cor:double_points}
If $P \in S$ is such that $
\phi_P(\Delta_P)= \{i,j\}$ 
and $\alpha = \sum_k \alpha_k \sigma_k, \; \beta = \sum_k \beta_k \sigma_k$ with $\alpha \curlywedge \beta= 0$, then either $\left| \array{cc} \beta_i & \alpha_i \\
\beta_j & \alpha_j \endarray \right| =0$
or $\mu_P(\delta_i, \cC_j) = \mu_P(\delta_j, \cC_i) = 0 \in \kk$ for all branches $\delta_i, \delta_j \in \Delta_P$, where $\phi_P(\delta_i) = i$ and $\phi_P(\delta_j) = j$.
\end{corollary}

\begin{proof}
Assume $\mu_P(\delta_i, \cC_j) \in\kk^*$ for some $\delta_i\in\Delta_P$ for which $\phi_P(\delta_i)=i$. Then, the hypothesis 
$\alpha \curlywedge \beta= 0$ and~\eqref{eq:det} imply the following resonance condition
$$
\left| \array{cc} \beta_i & \alpha_i \\
\mu_P(\delta_i,\cC_j)\beta_j & \mu_P(\delta_i,\cC_j)\alpha_j \endarray \right|= 
\mu_P(\delta_i,\cC_j)
\left| \array{cc} \beta_i & \alpha_i \\
\beta_j & \alpha_j \endarray \right| = 0,
$$
which proves $\left| \array{cc} \beta_i & \alpha_i \\ \beta_j & \alpha_j \endarray \right| = 0$.
\end{proof}

\section{Modular inequalities: upper bounds}
\label{sec:generalcase}
Consider as above the affine complement $X$ of a projective curve $\cC\subset \PP^2$. Consider also a rational character $\rho:\pi_1(X)\to \CC^*$ given by $\sigma_i\mapsto \xi^{m_i}_k$, where $\xi_k$
is a primitive $k$-root of unity and $k$ is a prime power $k=p^s$. Denote $\kk=\ZZ_p$.
Consider $(A^*_{W,\kk},\cdot\wedge\w_{\rho})$ the Combinatorial Aomoto complex associated with $W=W_{\cC}$ with $\kk$ coefficients, where $\omega_\rho:=\sum_i m_i\sigma_i\in A^1_{W,\kk}$. 
As an immediate consequence of Theorem \ref{thm:thmC_PS}, Remark \ref{rem:proj/affine}, and Theorem \ref{thm:resonance}, one has the following result.

\begin{thm}
\label{thm:inequality}
$\dim_\CC H^1(X;\CC_\rho)\leq \dim_\kk H^1(A^*_{W,\kk},\cdot\wedge\w_\rho)$.
\end{thm}

\begin{remark}
\label{rem:zariski_classic_result}
Note that Theorem \ref{thm:inequality} implies that the roots of the Alexander polynomial of an irreducible plane curve $\cC$ cannot be  primitive roots of unity of order $p^s$, $p$ prime number, $s \geq 1$. This recovers a classical result of Zariski from \cite{Z}, restated in modern terms and re-proved in  \cite[Proposition 2.1]{BD}.
\end{remark}

\subsection{A criterion for triviality}
\label{subsec:triv_crit}

\begin{definition}
Consider a projective curve $\cC$ and its weak combinatorics $W_{\cC}$. Assume there is a point $P\in\Sing(\cC)$ such that $\phi_P(\Delta_P)= \{i,j\}$ and $\mu_P(\delta,\cC_j)\in\kk^*$ for some branch with $\phi_P(\delta)=i$. Such a point will be called \emph{a $p$-transversal point of $\cC$}. Two components $\cC_i$ and $\cC_j$ are called \emph{$p$-transversal at a point} if there is a $p$-transversal point $P\in\cC_i\cap\cC_j$ of~$\cC$. 
\end{definition}
As an immediate consequence of Corollary~\ref{cor:double_points},
one has the following basic property of $p$-transversal points with respect to resonance varieties.
\begin{lemma}
\label{lemma:ptransversal}
Let $\cC_i$ and $\cC_j$ two $p$-transversal components at a point in $\cC$.
Consider $\alpha,\beta \in A^1_{W,\kk}$, $\alpha=\sum_k \alpha_k\sigma_k$, and 
$\beta=\sum_k \beta_k\sigma_k$. If $\alpha\curlywedge\beta=0$, then 
$$
\left|
\array{cc}
\alpha_i & \beta_i\\
\alpha_j & \beta_j
\endarray
\right|= 0\in\kk.
$$
\end{lemma}


In the special case of ordinary singularities, the resonance condition becomes simpler to describe. Recall that an ordinary singularity $P\in S$ is given locally as the intersection of smooth branches that intersect pairwise transversally, that is, $\mu_P(\delta_i,\delta_j)=1$ if $i\neq j$.

\begin{proposition}
\label{prop:ordinary_point}
Let $ P\in\cC_{i_1} \cap \ldots \cap \cC_{i_s} \subset {\rm Sing}(\cC)$ be an ordinary singularity such that $\#\Delta_P=s$ and $\w_1,\beta \in A^1_{W,\kk}$, $\w_1=\sum_i \sigma_i$, and 
$\beta=\sum_i \beta_i\sigma_i$. If $\w_1\curlywedge\beta=0$, then either 
\begin{equation}
\label{eq:equality}
\beta_{i_1} = \dots = \beta_{i_s} \,\, \textrm{ if } s\in\kk^*,
\end{equation}
or 
\begin{equation}
\beta_{i_1} + \dots + \beta_{i_s}=0\in\kk \textrm{ if } s=0\in\kk.
\end{equation}
\end{proposition}
\begin{proof}
Since $\w_1\curlywedge\beta=0$, the resonance condition~\eqref{eq:det} at $P$ becomes
$\left| \array{cc} \beta_i & 1 \\
\beta_{i_1}+\dots+\beta_{i_s} &  s \endarray \right|=0$ for all $i\in\{i_1,\dots,i_s\}$.
\end{proof}

The purpose of the rest of this section is to describe a reduction method that can decide the triviality of 
$H^1(A^*_{W_\cC,\kk},\cdot\curlywedge\alpha)$ for a given $\alpha=\sum_i \alpha_i\sigma_i\in A^1_{W_\cC,\kk}$. 

One can use Lemma~\ref{lemma:ptransversal} repeatedly when the components of a curve are pairwise $p$-transversal at a point, in order
to calculate the cohomology of $(A^*_{W_\cC,\kk},\cdot\curlywedge\alpha)$.
For technical reasons, we sometimes need $\alpha=\sum_i \alpha_i\s_i$ to be non-coordinate, that is, such that $\alpha_i\neq 0$ for all $i\in\mathbf{r}$. The following object will be useful to describe this method.

\begin{definition}
\label{def:ptransversalgraph}
For a curve $\cC$, one can define the \emph{$p$-transversality graph} as the non-oriented graph $\Gamma^p_{\cC}=(V,E)$ whose set of vertices $V=\mathbf{r}$ and such that $e=\{i,j\}\in E$ is an edge if $\cC_i$ and $\cC_j$ are $p$-transversal at a point.
\end{definition}

\begin{proposition}
\label{prop:pairwiseptransversal}
If $\Gamma^p_{\cC}$ is complete, then 
$H^1(A^*_{W_\cC,\kk},\cdot\curlywedge\alpha)=0$ for any non-zero $\alpha\in A^1_{W_\cC,\kk}$. 

Moreover, if $\Gamma^p_{\cC}$ is connected, then 
$H^1(A^*_{W_\cC,\kk},\cdot\curlywedge\alpha)=0$ for any non-coordinate $\alpha\in A^1_{W_\cC,\kk}$.
\end{proposition}

\begin{proof}
Assume $\beta=\sum_k \beta_k\sigma_k\in A^1_{W_\cC,\kk}$ such that $\beta \curlywedge\alpha=0$. Then, by Lemma~\ref{lemma:ptransversal}
one has $\rk (\beta,\alpha)=1$ and hence $\beta=\lambda\alpha$ for some $\lambda\in\kk$. The \emph{moreover} part is also immediate since
in a field 
$
\left|
\array{cc}
\alpha_{i_{n-1}} & \beta_{i_{n-1}}\\
\alpha_{i_{n}} & \beta_{i_{n}}
\endarray
\right|= 
\left|
\array{cc}
\alpha_{i_{n}} & \beta_{i_{n}}\\
\alpha_{i_{n+1}} & \beta_{i_{n+1}}
\endarray
\right|=0
$
implies
$
\left|
\array{cc}
\alpha_{i_{n-1}} & \beta_{i_{n-1}}\\
\alpha_{i_{n+1}} & \beta_{i_{n+1}}
\endarray
\right|=0.
$
\end{proof}

We will describe a recursive process to make this concept more general.
We will say the curve $\cC$ is $p$-transversal at $P\in \Sing(\cC)$ w.r.t. $\alpha$ if $\phi_P(\Delta_P)=\{i,j\}$ and 
$\mu_P(\delta,\cC_j)\alpha_j\in\kk^*$ for some $\delta\in\Delta_P$ such that $\phi_P(\delta)=i$. 
One can start a recursive process as follows. Consider the quotient $\mathbf{r}'=\mathbf{r}/\sim$ by identifying the previous components, say $[i]=[j]$ and denote by $\pi:\mathbf{r}\to \mathbf{r}'$ the projection map. One can say $\cC$ is $p$-transversal at another point $Q\in\Sing(\cC)$ w.r.t. $(\alpha,\pi)$ if $\pi(\phi_P(\Delta_P))=\{i',j'\}\subset\mathbf{r}'$ and $\sum_{\pi(j)=j'}\mu_Q(\delta,\cC_{j})\alpha_j\in\kk^*$ 
for some $\delta\in\Delta_P$ such that $\pi(\phi_P(\delta))=i'$. This process can be extended to any projection map.
The curve $\cC$ is called \emph{completely $p$-reductive w.r.t. $\alpha$} if one can apply this recursive method until the quotient 
set $\mathbf{r}'$ eventually contains only one element.

\begin{example}
	\label{exam:TransversalCurves}
	Consider for instance that $\cC = \cC_1 \cup  \cC_2$ is the union of two curves  $\cC_1,  \cC_2$ that intersect transversely. 
	Then $\cC$ is completely $p$-reductive for any $p$ w.r.t. any non-coordinate~$\alpha\in A^1_{W_\cC,\kk}$. More generally, if the graph $\Gamma^p_\cC$ is connected, then $\cC$ is completely $p$-reductive w.r.t. any non-coordinate~$\alpha\in A^1_{W_\cC,\kk}$.
\end{example}

\begin{proposition}
	\label{prop:completelypreductive}
	Let $\cC$ be a completely $p$-reductive curve w.r.t. a non-coordinate $\alpha\in A^1_{W_\cC,\kk}$. 
	Then the cohomology jumping loci of its Combinatorial Aomoto complex $(A^*_{W_\cC,\kk},\cdot\curlywedge\alpha)$ is trivial, i.e.,
	$$
	H^1(A^*_{W_\cC,\kk},\cdot\curlywedge\alpha)=0.
	$$
\end{proposition}

\begin{proof}
Consider $\beta=\sum_i \beta_i\sigma_i\in A^1_{W_\cC,\kk}$ such that $\beta\curlywedge\alpha=0$.
Let's say the curve $\cC$ is $p$-transversal at $P\in \Sing(\cC)$ w.r.t. $\alpha$.
By Corollary~\ref{cor:double_points}, this implies the existence of $\lambda\in\kk$ such that $(\beta_i,\beta_j)=\lambda(\alpha_i,\alpha_j)$.
Assume the following inductive hypothesis: for any element $i'\in\mathbf{r}'$ there exists $\lambda\in\kk$ such that 
$\beta_i=\lambda\alpha_i$ for all~$\pi(i)=i'$. One can prove the following. Assume $\cC$ satisfies the inductive hypothesis
w.r.t. $(\alpha,\pi)$ and $\beta\curlywedge\alpha=0$. If $\cC$ is $p$-transversal at $P\in\Sing(\cC)$ w.r.t. $(\alpha,\pi)$ and denote by $\pi'$
the resulting projection map, then $\cC$ satisfies the inductive hypothesis w.r.t. $(\alpha,\pi')$. In order to prove this claim, note that by hypothesis
$\pi(\phi_P(\Delta_P))=\{i',j'\}\subset\mathbf{r}'$ and there exists $\delta\in\Delta_P$ with $\phi_P(\delta)=i_0$, $\pi(i_0)=i'$ such that $\sum_{\pi(j)=j'}\mu_Q(\delta,\cC_{j})\alpha_j\in\kk^*$. By condition~\eqref{eq:det} one has
\begin{equation}
	\label{eq:preduction}
\array{ll}
	0 & = \left| \array{cc} \beta_{i_0} & \alpha_{i_0} \\
	\sum_k \mu_P(\delta,\cC_k)\beta_k  & \sum_k \mu_P(\delta,\cC_k)\alpha_k \endarray \right|\\
	\\
    & =
	\left| \array{cc} \beta_{i_0} & \alpha_{i_0} \\
\sum\limits_{\pi(i)=i'} \mu_P(\delta,\cC_i)\beta_i+\sum\limits_{\pi(j)=j'} \mu_P(\delta,\cC_j)\beta_j & 
\sum\limits_{\pi(i)=i'} \mu_P(\delta,\cC_i)\alpha_i+\sum\limits_{\pi(j)=j'} \mu_P(\delta,\cC_j)\alpha_j \endarray \right|.
\endarray
\end{equation}
By induction hypothesis $\beta_i=\lambda_1\alpha_i$ (resp. $\beta_j=\lambda_2\alpha_j$) for any $i\in\mathbf r$ such that $\pi(i)=i'$
(resp. for any $j\in\mathbf r$ such that $\pi(j)=j'$). Hence
\begin{equation}
	\label{eq:preduction2}
	\array{ll}
	0 & = 
	\left| \array{cc} \lambda_1\alpha_{i_0} & \alpha_{i_0} \\
	\lambda_1\sum\limits_{\pi(i)=i'} \mu_P(\delta,\cC_i)\alpha_i+\lambda_2\sum\limits_{\pi(j)=j'} \mu_P(\delta,\cC_j)\alpha_j & 
	\sum\limits_{\pi(i)=i'} \mu_P(\delta,\cC_i)\alpha_i+\sum\limits_{\pi(j)=j'} \mu_P(\delta,\cC_j)\alpha_j \endarray \right| \\
	\\
	& =
	\alpha_{i_0}\sum\limits_{\pi(i)=i'} \mu_P(\delta,\cC_i)\alpha_i \left| \array{cc} \lambda_1 & 1 \\
	\lambda_1& 1 \endarray \right| +
	\alpha_{i_0}\sum\limits_{\pi(j)=j'} \mu_P(\delta,\cC_j)\alpha_j \left| \array{cc} \lambda_1 & 1 \\
	\lambda_2 & 1\endarray \right|.
	\endarray
\end{equation}
Since $\alpha_{i_0}\neq 0$ and  $\sum_{\pi(j)=j'}\mu_Q(\delta,\cC_{j})\alpha_j\in\kk^*$ by hypothesis, this implies $\lambda_1=\lambda_2$
and hence $\cC$ satisfies the inductive hypothesis, that is, $\beta_k=\lambda_1\alpha_k$ for all $\pi(k)=i'$ or $\pi(k)=j'$.
\end{proof}

\begin{remark}
Note that the technical condition in Proposition~\ref{prop:completelypreductive} about a non-coordinate $\alpha$ can be avoided in practice by considering the subcurve of $\cC$ whose components correspond with the non-zero coefficients of $\alpha$. For instance, if $\alpha_i=0$ for some $i\in\mathbf{r}$ but $\sum_{k=1}^r d_k\alpha_k\neq 0$ one can use 
the resonance condition~\eqref{eq:lambda_i} to show that $\beta\curlywedge\alpha=0$ implies $\beta_i=0$. Hence, one can consider $\tilde \cC=\cC_1\cup\cdots\cup\hat\cC_i\cup\cdots\cup\cC_r$ the subcurve of $\cC$ resulting from removing the component $\cC_i$ from $\cC$. Denote by $\hat\alpha$ the 1-cochain in $A^*_{W_{\hat\cC},\kk}$. 
It is immediate that $H^1(A^*_{W_\cC,\kk},\cdot\curlywedge\alpha)=H^1(A^*_{W_{\hat\cC},\kk},\cdot\curlywedge\hat\alpha)$. Using this reduction method one ends up with a non-coordinate cochain for a subcurve of~$\cC$ and 
Proposition~\ref{prop:completelypreductive} could be applied.

This reduction process can also be used if for instance if $\sum_k \mu_P(\delta,\cC_k)\alpha_k\neq 0$ for some $P\in\Sing(\cC)$, 
and some $\delta\in\Delta_P$ with~$\phi_P(\delta)=i$.
\end{remark}

\begin{proposition}
Let $\cC$ be a curve that is the union of $r \geq 3$ fibers of a pencil of curves. Then $\cC$ is not completely $p$-reductive for any $p$.
\end{proposition}
\begin{proof}
Let  $\cC= \overline \cC_1 \cup  \dots \cup \overline \cC_r$ be the decomposition of $\cC$ as union of members of a pencil. Notice that the base points of the pencil are never $p$-transversal since, for any base point $P$, $|\phi_P(\Delta_P)| \geq 3$. In particular, we cannot apply a $p$-reduction to two components of a combinatorics $\tilde W$ if one of the components intersects $\overline \cC_i$ and the other component intersects $\overline \cC_j$ with $i \neq j$. So, after applying all the possible $p$-reductions, the combinatorics that we obtain has at least $r$ components. 
\end{proof}

\subsection{Twisted Milnor fibers and twisted Alexander polynomials}
\label{subsec:twisted}
Consider $\cC= \cC_1 \cup  \ldots \cup \cC_r$ the decomposition of $\cC$ in irreducible components, where $\cC_i$ is defined by the set of zeroes of an irreducible homogeneous polynomial $F_i$. Also consider a multiplicity list $\bar m=(m_1,\ldots,m_r)$ such that $\gcd \bar m=1$. This makes $\{F_1^{m_1}\cdots F_r^{m_r}=1\}\subset \CC^3$ a connected space called the \emph{twisted Milnor fiber} associated with the projective curve $\cC$, and it will be denoted as $\cF(\cC,\bar m)$. We will denote by $\Delta_{(\cC,\bar m)}$ the characteristic polynomial of the twisted Milnor fiber $\cF(\cC,\bar m)$, also called the {\it $\bar m$-twisted Alexander polynomial} associated with~$\cC$. 
Note that 
\begin{equation}
\label{eq:twisted_alexander_poly}
\Delta_{(\mathcal{C}, \bar m)}(t) = \prod_{k | \bar N} \phi_k ^{b_k}
\end{equation}
is a product of cyclotomic polynomials $\phi_k(t)$.
The exponents $b_k$ can also be computed as
\begin{equation}
\label{eq:local_betti_twist}
b_k = {\rm dim}_{\mathbb{C}}H^1(\tilde  X;\mathbb{C}_{\rho^{\bar N/k}}),
\end{equation}
where $\tilde X = \PP^2 \setminus \cC$,  $\bar N = d_1 m_1 + \dots +d_r m_r$, $d_i =\deg F_i$, and $\rho:\pi_1(\tilde{X})\to \CC^*$ is defined by $\sigma_i\mapsto \xi^{m_i}_{\bar N}$, for a primitive $\bar N$-root of unity $\xi_{\bar N}$ (see \cite{DS, Hir, Lib, Sak}). The case $\bar m = (1,  \ldots, 1)$ corresponds to the classical Alexander polynomial $\Delta_\cC$ described in \S\ref{ss:milnor&local_coeff}.

In this context, one can say a curve $\cC$ is completely $p$-reductive w.r.t. $\bar m=(m_1,\dots,m_r)$ if $\cC$ is completely $p$-reductive w.r.t. 
$\w_\rho := \sum_i m_i \sigma_i$.

It follows from Remark \ref{rem:proj/affine} that Theorem \ref{thm:inequality} gives an upper bound for the multiplicities $b_k$ from \eqref{eq:twisted_alexander_poly} in terms of the Combinatorial Aomoto complex.

\begin{thm}
\label{thm:completelypreductive}
Assume that $\cC$ is a completely $p$-reductive curve w.r.t. $\bar m=(m_1, \ldots ,m_r)$ where $\gcd \bar m=1$.
Then the primitive $k$-root of unity $\xi_k$ with $k=p^s$ is not a root of the characteristic polynomial $\Delta_{(\cC,\bar m)}$ of the twisted Milnor fiber $\cF(\cC,\bar m)$ associated with $\cC$ and $\bar m$.
\end{thm}

\begin{proof}
Since $\cC$ is completely $p$-reductive w.r.t. the non-coordinate cochain $\w_\rho$, by Proposition~\ref{prop:completelypreductive} one has 
$\dim_\kk H^1(A^*_{W_{\cC},\kk},\cdot \curlywedge \w_\rho)=0$, where  
 $\rho:\pi_1(\tilde X)\to \CC^*$ is the local system given by $\sigma_i\mapsto \xi^{m_i}_k$, $\w_\rho = \sum_i m_i \sigma_i$.
Theorem~\ref{thm:inequality} ends the proof.
\end{proof}

\begin{remark}
In~\cite[Theorem 34]{Oka}, Oka proves that any curve $\cC=\cC_1\cup\cC_2$ that is a transversal union of two curves $\cC_1$ and $\cC_2$ has a trivial Alexander polynomial. Note that 
Proposition~\ref{prop:pairwiseptransversal} and Theorem~\ref{thm:completelypreductive} give this result only for roots of unity whose order is a prime power. On the other hand, it extends this result to certain rank-one twisted Alexander polynomials (known as $\theta$-polynomials or Oka polynomials).
\end{remark}

\subsection{Examples}
\label{subsec:milnor}
We can use the combinatorial objects  defined above to estimate the non-triviality of certain factors of the Alexander polynomial $\Delta_{\cC}(t)$ of the Milnor fiber associated to a given projective curve $\cC$. More precisely, to find upper bounds for the multiplicities of the roots of $\Delta_{\cC}(t)$ that are primitive roots of unity of order a prime power, via Theorem \ref{thm:inequality}.

We will illustrate this with a series of examples that come from Halphen pencils of Hesse-type and index $2$, as defined in \cite{DLPU}, see Remark \ref{rem:halphen_pencil} for a refresher on the definition of Halphen pencils. Here we refer to \cite{DPS} for algorithmic computations of the  Alexander polynomial of the Milnor fiber corresponding to these examples.

\begin{example}
\label{ex:thm6.2_DPS}
Consider the conic-line arrangements $\mathcal{B} \subset \mathcal{A}$ defined by $$\mathcal{B} = \cC_4 \cup \dots \cup \cC_{12}, \; \mathcal{A} = \mathcal{B} \cup L_1 \cup L_2 \cup L_3,$$

where 
$$
\array{ccc}
L_1: {x=0}, & L_2: {y=0}, & L_3: {z=0},
\endarray
$$
$$
\array{ll}
\cC_4 : x^2+xy+y^2+xz+yz+z^2 =0,  \\
\cC_5 : x^2+\bar\xi xy+\xi y^2+\xi xz+yz+\bar\xi z^2=0,
&
\cC_7:  x^2+\bar\xi xy+\xi y^2+\bar\xi xz+\xi yz+\xi z^2 = 0, \; 
\\
\cC_9: x^2+xy+y^2+\xi xz+\xi yz+\bar\xi z^2 =0, 
&
\cC_{11}: x^2+\xi xy+\bar\xi y^2+xz+\xi yz+z^2 =0, \;
\\
\cC_6=\bar\cC_5,\quad 
\cC_8=\bar\cC_7,&
\cC_{10}=\bar\cC_9,\quad 
\cC_{12}=\bar\cC_{11},
\endarray
$$
where $\xi^2+\xi+1=0$, and the notation $\bar\cC$ refers to a curve whose points are complex conjugated to those of $\cC$.

Since $\deg(f_\mathcal{A})=21$ and it has 12 irreducible components, the general form of the Alexander polynomial of the Milnor fiber associated to $\mathcal{A}$ is 
$$
 \Delta_{\mathcal{A}}(t) = (t-1)^{11}\phi_3^{b_3}\phi_7^{b_7}\phi_{21}^{b_{21}}.
 $$
The singular locus $\Sing(\mathcal{A})$ consists of $12$ ordinary double points and $9$ ordinary points of multiplicity $7$.
 Let us describe in detail the aspects of the combinatorics of $\mathcal{A}$ relevant for our computations. To begin with, $\mathcal{A}$ is the reduced union of four reducible members of a pencil of sextics. This pencil has a (unique) member which is a cuspidal sextic with $9$ cusps, which is dual to a smooth cubic.  
 If $f_i$ denotes the equation of $\cC_i$, note that $$f_4f_5f_6+\bar\xi f_7f_9f_{11}+\xi f_8f_{10}f_{12}=0.$$ Hence, $\cC_4 \cup \cC_5 \cup \cC_6$, $\cC_7 \cup \cC_9 \cup \cC_{11}$, and $\cC_8 \cup \cC_{10} \cup\cC_{12}$ are three such members. Also, $$(\xi-1)f_4f_5f_6-(\xi-1)f_7f_9f_{11} +9x^2y^2z^2=0,$$ which implies that $L_1\cup L_2\cup L_3$ is a non-reduced member of this pencil with multiplicity $2$. 
 The generic sextic of this pencil has 9 nodes at the base points of the pencil, corresponding to the 9 ordinary points of multiplicity 7 of $\cA$. Outside of these singular points, the 12 remaining ordinary double points of $\cA$ are distributed in 3 nodes for each of the four special sextics described above. Every pair of irreducible components for each special sextic contains an ordinary double point. This shows that each one of the four special fibers is completely $p$-reductive for any $p$ w.r. t. any non-coordinate $\w\in A^1_{W_\cC,\kk}$. 
 Through the points of multiplicity $7$, i.e., the base points of the pencil, there is an intersection of two conics from each of the first three special fibers and one line from the fourth special fiber.
 
 Since $\cA$ only has only ordinary singularities, we can use Proposition \ref{prop:ordinary_point} to compute upper bounds for the exponents $b_3, b_7$, via Theorem \ref{thm:inequality}. In fact, since this curve has only ordinary multiple points of multiplicities $2$ and $7$, it is straightforward from Proposition \ref{prop:ordinary_point} that $\dim_{\mathbb{Z}_p} H^1(A^*_{W_{\mathcal{A}},\mathbb{Z}_p},\cdot\wedge\w_1)=0$, for $p\neq 2,7$, where $\w_1 = \sum_i \sigma_i$, 
 and thus $b_3=0$. Again, using Proposition \ref{prop:ordinary_point} we compute $\dim_{\mathbb{Z}_7} H^1(A^*_{W_{\mathcal{A}},\mathbb{Z}_7},\cdot\wedge\w_1)=2$, so necessarily $b_7 \leq 2$.
 
Let us consider now the subarrangement $\mathcal{B}$ of $\cA$ defined above after deleting the lines from $\cA$. 
 $\mathcal B$ has degree 18 and 9 irreducible components, hence
 $$
 \Delta_{\mathcal{B}}(t) = (t-1)^{8}\phi_2^{b_2}\phi_3^{b_3}\phi_{6}^{b_{6}}\phi_9^{b_9}\phi_{18}^{b_{18}}.
 $$
Note that $\mathcal B$ has $9$ nodes and $9$ ordinary points of multiplicity $6$ in the singular locus $\Sing(\mathcal{B})$. 
Taking into account the weak combinatorics of this arrangement, Proposition \ref{prop:ordinary_point} immediately implies 
$$\dim_{\mathbb{Z}_2} H^1(A^*_{W_{\mathcal{B}},\mathbb{Z}_2},\cdot\wedge\w_1)=2,$$ and
$$\dim_{\mathbb{Z}_3} H^1(A^*_{W_{\mathcal{B}},\mathbb{Z}_3},\cdot\wedge\w_1)=1,$$
so necessarily $b_2 \leq 2$ and $b_3, b_9 \leq 1.$
 One can compare our computations with \cite[Theorem 6.2]{DPS}, where the Alexander polynomials of the Milnor fiber of $\mathcal{A}, \mathcal{B}$ are computed via \verb}SINGULAR}.
\end{example}
\begin{example}
\label{ex:DPS_thm6.4}
Let $\mathcal{B} \subset \mathcal{A}$ be the conic-line arrangements: 
$$\mathcal{B} = \cC_4 \cup ... \cup \cC_{12}, \; \mathcal{A} = \mathcal{B} \cup L_1 \cup L_2 \cup L_3,$$ where 
$$
\array{lll}
L_1: x+y+z=0, & 
L_2: x+\xi y+\bar\xi z = 0, & 
L_3: x+\bar\xi y+\xi z  = 0, \\ 
\cC_4: y^2-xz=0, & 
\cC_5: x^2-yz = 0, & 
\cC_6 : xy-z^2 = 0, \\ 
\endarray
$$
$$
\array{ll}
\cC_7 : x^2-xy+\xi y^2-\xi xz-yz+z^2=0, & \cC_8=\bar\cC_7, \\ 
\cC_9: x^2-\xi xy+\xi y^2-\xi xz-yz+\xi z^2 = 0, & \cC_{10}=\bar\cC_9, \\
\cC_{11}: x^2-\xi xy+y^2-xz-yz+\xi z^2 = 0, & \cC_{12}=\bar\cC_{11}
\endarray
$$
where $\xi^2+\xi+1=0$. 
Note that 
$$9f_4f_5f_6+(2\xi + 1)f_7f_{10}f_{11}-(2\xi + 1)f_8f_{9}f_{12}=0,$$
$$9f_4f_5f_6+(\xi + 2)f_7f_{10}f_{11}-(\xi + 2)(f_1f_2f_3)^2=0,$$
where $L_{i} = V(f_{i})$ with $i \in \{1,2,3\}$ and $C_{j} = V(f_{j})$ with $j \in \{4, \ldots , 12\}$.
 As in the previous example, $\mathcal{A}$ is the reduced union of four singular members of a pencil of sextics: three members are unions of three smooth conics and the fourth is the union of three double lines.
Let us briefly describe some aspects of the weak combinatorics of $\mathcal{A}$. It has nine nodes, corresponding to each pair in the triples $(\cC_4, \cC_5, \cC_6)$, $(\cC_7, \cC_{10}, \cC_{11})$, $(\cC_8, \cC_{9}, \cC_{12})$, and three singular points of multiplicity $11$, namely $P=[\xi:\bar\xi:1]$, $Q=[\bar\xi:\xi:1]$, and $R=[1:1:1]$. The points $P,Q,R$ are the base points of the pencil, and we have $L_1 = \overline{PQ}, L_2 = \overline{PR}, L_3=\overline{QR}$, and $\{P,Q,R\} \subset \cC_i$ for every $i \in \{4,\ldots,12\}$.

The singularities $P,Q,R$ are not ordinary. Take, for instance, point $P$. In the notations from the beginning of Section 
 \ref{ss:weak_comb}, we have
 $$\mu_P(\delta_i, \delta_j) =2 \; \Longleftrightarrow \; \{i,j\} \in \{\{4,10\}, \{5, 9\}, \{6, 7\}, \{4,12\}, \{5, 8\}, \{6, 9\}, \{7, 9\}, \{10, 12\}, \{8, 11\}\}.$$
In all the other cases, $\mu_P(\delta_i, \delta_j) =1$. Due to the symmetry of the arrangement $\mathcal{A}$, the combinatorics around the points $Q,R$ mirrors the one of $P$, so we will omit the details for these two singularities.
 
 Since the degree of the reduced homogeneous polynomial that defines $\mathcal{A}$ is $21$ the Alexander polynomial of its Milnor fiber is of type
$$
 \Delta_{\mathcal{A}}(t) = (t-1)^{11}\phi_3^{b_3}\phi_7^{b_7}\phi_{21}^{b_{21}}.
 $$
 In order to compute upper bounds for the multiplicities of the primitive roots whose  order is equal to a power of a prime, i.e. $b_3, b_7$, we turn to Theorem \ref{thm:inequality}. 
Let $(A^*_{W_{\mathcal{A}},\ZZ_p},\cdot\wedge\w_1)$ be the Combinatorial Aomoto complex associated with the weak combinatorics of $\mathcal{A}$, and $\w_1:=\sum_i \sigma_i\in A^1_{W_{\mathcal{A}},\ZZ_p}$. Then for $p=3,7$ one has 
$$b_p \leq \dim_{\mathbb{Z}_p} H^1(A^*_{W_{\mathcal{A}},\mathbb{Z}_p},\cdot\wedge\w_1).$$
The right hand side of this inequality is easily computable using the linear system derived from \eqref{eq:det}, and we get 
$$b_3 = 0, \; b_7 \leq 2.$$
We proceed the same way with the (sub)arrangement $\mathcal{B}$. In this case $$
 \Delta_{\mathcal{B}}(t) = (t-1)^{8}\phi_2^{b_2}\phi_3^{b_3}\phi_{6}^{b_{6}}\phi_9^{b_9}\phi_{18}^{b_{18}}.
 $$
 and we obtain by Theorem \ref{thm:inequality} and \eqref{eq:det} that 
 $$b_2 \leq 2, \; b_3, b_9 \leq 1.$$
 The Alexander polynomials of $\mathcal{A}, \mathcal{B}$ are computed via \verb}SINGULAR} in \cite[Theorem 6.4]{DPS}.\\
\end{example}
\begin{example}
\label{ex:Thm_5.4_DPS}
Consider the family $\cC(\lambda)$ of non-degenerate Hesse arrangements of conics whose Alexander polynomials are computed in \cite[Theorem 5.4]{DPS}. It consists of $12$ smooth conics, organized as four reduced members in a Halphen pencil of index $2$. Each member thus consists in a union of three smooth conics. The pairs of conics inside each union intersect in $4$ distinct points, one of these points is an ordinary double point of the curve, and the remaining three intersection points, considered as singularities of the curve, are ordinary points of multiplicity $8$. In total, the singular locus $\Sing(\cC(\lambda))$ consists of $12$ ordinary double points and $9$ ordinary points of multiplicity $8$.
For more details on the combinatorics of $\cC(\lambda)$, which is a realization of the abstract configuration of conics and points $(12_6, 9_8)$,  see \cite{DLPU}. For an arbitrary arrangement in this family, the only possibly non-trivial multiplicities of the primitive roots of the Alexander polynomial whose order is equal to a power of a prime are $b_2, b_3, b_4, b_8$.
If we take into account the weak combinatorics of these arrangements, then by Proposition \ref{prop:ordinary_point} and Theorem \ref{thm:inequality} we get immediately 
$$b_3=0.$$ As for the rest of the multiplicities, the same two results imply $$b_2, b_4, b_8 \leq 3.$$
\end{example}

\section{Pencils and modular inequalities: lower bounds}
\label{sec:pencils}

\subsection{A definition of quasi fiber--type curves}
\label{subsubsect:quasi}
Assume that a curve $\cC$ is such that its complement $\tilde X=\PP^2\setminus\cC$ admits a surjection
$f:\tilde X\to S$ onto a hyperbolic orbifold $S=\PP^1_{(r,\bar{m})}$, that is, $\chi^{\orb}(S)<0$. 

In that case there exists a partition $\cC=\cC_1\cup\cC_2\cup \ldots \cup\cC_{r}$ with $\cC_i=\{F_i=0\}$, some $m_{i,j}>0$ for which $F_i=\prod_j F_{i,j}^{m_{i,j}}$ ($F_{i,j}$ irreducible), and $H_i$ such that $G_i=F_iH^{m_i}_i$,
where $\bar m = (m_1, \ldots, m_r)$, $d=\deg G_{i}$ for all $i \in \{1,\ldots,r \}$, satisfying that $G_1, \ldots,G_{r}$ are aligned in the projective space of curves of degree~$d$. 

For convenience, we assume $m_{i}=0$ if and only if $H_i$ is constant. If, in addition, the pencil is \emph{primitive}, that is, its generic member is irreducible, we will refer to this type of curve as \emph{quasi fiber--type}. 

If $\varphi$ is the collection $(F_i,m_{i,j},H_i,\bar m)$ then we will simply refer to its quasi fiber--type decomposition as $\cC_\varphi$. For simplicity, we will assume $\gcd(m_{i,j},j)=1$ for every $i=1,\ldots,r$. The following results can easily be extended without this condition, but the statements become more involved.

We also say a quasi fiber--type curve has \emph{index} $k\geq 1$ if the pencil containing $G_1,\ldots,G_r$ contains also a multiple member~$G^k$ outside~$G_i$ and $k|m_i$ for all $i \in \{1, \ldots ,r \}$.

Note that the existence of a quasi fiber--type curve of index $k$ implies the existence of a surjection from the fundamental group of its complement $\tilde X$ as follows
\begin{equation}
\label{eq:proper-surj}
\pi_1(\tilde X)\overset{f_*}\surj \pi^{{\rm orb}}_1(\PP^1_{\bar{m}'}),
\end{equation}
where $\bar m'=(k,m_1,\ldots,m_r)$, see~\cite{BCM} for details.

\begin{remark}
\label{rem:halphen_pencil}
A quasi fiber--type curve of index $k$ for which $m_i=0$ with $i \in \{1, \ldots ,r\}$ and $k>1$ is also known as a \emph{fiber--type} or \emph{pencil--type} arrangement associated with a Halphen pencil of index $k$, see~\cite[Section 2]{DLPU}. A (quasi) fiber--type curve of index $k=1$ is simply a (quasi) fiber--type curve.

Also note that a quasi fiber--type curve can have more than one index. For instance, if $F_{\lambda,\mu}:=\lambda F_p^q+\mu F_q^p$, where $F_k$, is a homogeneous polynomial of degree $k$ in three variables, then $\cC=\cC_1\cup \ldots \cup\cC_r$, where $\cC_i=\{F_{\lambda_i,\mu_i}=0\}$ and $[\lambda_i:\mu_i]\in \PP^1\setminus\{[0:1],[1:0]\}$ with $i=1, \ldots ,r$ are different points, is a quasi fiber--type curve of index $p$ or $q$.
\end{remark}

\subsubsection{Quasi fiber--type structures and resonance varieties}

Let $\cC_\varphi$ be a quasi fiber--type curve as above, we denote by $\sigma_{i,j}$ a meridian of the irreducible component $\cC_{i,j}=\{F_{i,j}=0\}$.
We will show how a quasi fiber--type decomposition of index $p$ in a curve implies the existence of positive dimensional components in the resonance variety of characteristic~$p$.

\begin{definition}
\label{def:const_on_fiber}
Let $\cC_\varphi$ be a quasi fiber--type curve defined by the union of $r$ quasi fibers $\cC_1, \dots, \cC_r$. An element 
$\w:= \sum_{i,j} m_{i,j} \alpha_{i,j} \sigma_{i,j}\in A^1_{W_\cC,\kk}$ is said to be 
\emph{associated with $\cC_\varphi$} if $\alpha_{i,j}$ is independent of~$j$ for any fixed~$i$.
\end{definition}

\begin{remark}
Consider $\cR_{\varphi,\kk}\subset A^1_{W_\cC,\kk}$ the set of all 1-cochains associated with $\cC_\varphi$. 
Note that the condition $\gcd(m_{i,j}: j)=1$ for all $i=1,...,r$ implies 
\begin{equation}
\label{eq:R_phi}
\dim_\kk \cR_{\varphi,\kk} = r.
\end{equation}
\end{remark}

\begin{thm}
\label{thm:fibered}
Consider $\cC_\varphi=\cC_1\cup\cdots\cup\cC_r$ a quasi fiber--type curve, $\mathrm{char}(\kk)=p$, 
such that $m_i=0\in\kk$ for all $i=1,...,r$. Assume $\omega:=\sum_{i,j} m_{i,j}\alpha_{i}\sigma_{i,j}\in\cR_{\varphi,\kk}$ is associated with $\cC_\varphi$. Then, 
\begin{enumerate}
    \item\label{thm:fibered:1} If $\sum_i \alpha_i=0\in\kk$,
then $\omega\in \cR_{r-2}(W_\cC,\kk)$, that is,
\begin{equation}
\label{eq:fibered:general}
h^1(A^*_{W_{\cC},\kk},\cdot\wedge\w)\geq r-2 
\end{equation}
    \item\label{thm:fibered:2} If $\cC_\varphi$ has index $k=0\in\kk$,
then $\cR_{\varphi,\kk}\subset \cR_{r-1}(W_\cC,\kk)$, that is,
\begin{equation}
\label{eq:fibered:indexp}
h^1(A^*_{W_{\cC},\kk},\cdot\wedge\w)\geq r-1.
\end{equation}
\end{enumerate}
\end{thm}

\begin{proof}
We use the notation $G_1, \ldots ,G_r,G^k$ for the polynomials in the pencil given by $\cC_\varphi$ as described at the beginning of this section. Let $\cC_{i,j}: F_{i,j}=0$. 
Given a point $P\in \cC$ and a branch $\delta$ of $\cC_i$, we denote by $\mu_P(\delta,G_i)$ the multiplicity of the intersection of $\delta$ with $G_i$ at $P$. Note that 
$\mu_P(\delta,G^k)=k\mu_P(\delta,G)$.

The singular locus $\Sing(\cC)$ can be decomposed into two disjoint sets $\Sing(\cC)=B\cup N$ according to whether or not they are base points. In particular, we have the following.
\begin{enumerate}
\item \label{item:base}
If $P\in B$ (base points), then $G_i(P)=0$ for all $i \in \{1, \ldots ,r\}$. Moreover, given a branch of $\cC_i$ at $P$ (i.e. $\phi_P(\delta)=i$), then 
\begin{equation} \label{eq:mu_delta}
\mu_P(\delta,G_j)=
\mu_P(\delta,F_j)+\mu_P(\delta,H_j^{m_j})=
\sum_k m_{j,k}\mu_P(\delta,\cC_{j,k})+m_j\mu_P(\delta,H_j)=\mu_\delta
\end{equation}
has the same value regardless of $j\neq i$. Moreover, since $G^k$ is a fiber in the linear system, one has $\mu_\delta=k\mu_P(\delta,G)=0\in\kk$. Since $m_j=0\in\kk$ this implies 
\begin{equation}\label{eq:quasift}
\sum_k m_{j,k}\mu_P(\delta,\cC_{j,k})=0\in\kk \textrm{ for any } j\neq i.
\end{equation}
\item \label{item:nonbase}
If $P\in N$ (non-base points) is such that $\phi_P(\Delta_P)=\{i\}$, then $F_j(P)=0$ if and only if $i=j$.
\end{enumerate}
We will use Conditions~\eqref{eq:lambda_i} and \eqref{eq:det} 
to show that $\w_1\curlywedge \w_2=0$ for any two 1-forms $\w_1,\w_2\in \cR_{\varphi,\kk}$ associated with $\cC_\varphi$, say $\w_1 = \sum_{i,j} m_{i,j}\alpha_{i}\sigma_{i,j}, \; \w_2 = \sum_{i,j} m_{i,j}\beta_{i}\sigma_{i,j}$. 
As in~\eqref{eq:awb} we use the notation 
$$
\w_1\curlywedge\w_2=
\sum_{\delta,P}\lambda_{\delta,P} \bar\psi_{P}^\delta+
\sum_{i,j} \lambda_{i,j}\bar\psi_\infty^{i,j}.
$$
We will show $\lambda_{\delta,P}=\lambda_{i,j}=0$ for all $\delta\in\Delta_P$, all $P\in S$, and $i=1,...,r$.
Consider $P\in N$, then 
condition~\eqref{eq:det} becomes
\begin{equation} \label{eq:non_base_P}
\lambda_{\delta,P}=    
\left|
\array{cc}
\alpha_{i,j} & \beta_{i,j}\\
\sum_s \mu_P(\delta,\cC_{i,s})m_{i,s}\alpha_{i,s} & \sum_s \mu_P(\delta,\cC_{i,s})m_{i,s}\beta_{i,s}\\
\endarray
\right|=0.
\end{equation}
since $\alpha_{i,j}=\alpha_{i,s}=\alpha_i$
and 
$\beta_{i,j}=\beta_{i,s}=\beta_i$ by hypothesis.

Assume now that $P\in B$. In that case, condition~\eqref{eq:det} is also satisfied 
since 
\begin{equation} \label{eq:q_fibr_mod_p}
    \sum_{i,k} \mu_P(\delta,\cC_{i,k})m_{i,k}\alpha_{i,k}=
\mu_\delta \sum_i \alpha_i = 0\in\kk.
\end{equation}
The equalities follow from conditions \eqref{eq:mu_delta} and \eqref{eq:quasift} since $P\in B$. 

Condition~\eqref{eq:lambda_i} becomes
\begin{equation} \label{eq:non_base_P_lambda}
\lambda_i=
\left|
\array{cc}
\alpha_{i,j} & \beta_{i,j}\\
\sum_{s,t} d_{s,t}m_{s,t}\alpha_{s,t} & \sum_{s,t} d_{s,t}m_{s,t}\beta_{s,t}\\
\endarray
\right|=
    \left|
\array{cc}
\alpha_{i} & \beta_{i}\\
kd_G\sum_{s} \alpha_{s} & kd_G\sum_{s} \beta_{s}\\
\endarray
\right|=
0
\end{equation}
since either $k=0\in\kk^*$ or $\sum_{s} \alpha_{s}=\sum_{s} \beta_{s}$ by hypothesis.

The result on the dimension follows by~\eqref{eq:R_phi}.

Note that equality~\eqref{eq:q_fibr_mod_p} is also true if $\sum_i\alpha_i=0\in\kk$ and $m_i=0\in\kk$ which is the hypothesis of the first part. The dimension $r-2$ is given by~\eqref{eq:R_phi} and the extra condition $\sum_i\alpha_i=0\in\kk$.
\end{proof}

Under certain extra conditions on $p$, Theorem~\ref{thm:fibered} is sharp and it calculates the depth of the component determined 
by the canonical 1-form $$\w_\varphi:= \sum_{i,j} m_{i,j} \sigma_{i,j}\in A^1_{W_\cC,\kk}.$$
In order to state the result, we need some notation. Consider $\cC_\varphi=\cC_1\cup\dots\cup\cC_r$ a quasi fiber--type curve and 
$\w:= \sum_{i,j} m_{i,j} \alpha_{i} \sigma_{i,j}\in A^1_{W_\cC,\kk}$ associated with $\cC_\varphi$. For any subcurve $\cC_i$, we denote by
$\w_{i}:= \sum_{j} m_{i,j} \alpha_{i} \sigma_{i,j}\in A^1_{W_{\cC_i},\kk}$ the projection of $\w_{\varphi}$ to $A^1_{W_{\cC_i},\kk}$.
Since $\w$ is associated with $\cC_\varphi$, one can write $\w_{i}:= \alpha_{i} \sum_{j} m_{i,j} \sigma_{i,j}$.
Also, note that one can formaly recover $\w$ as  $\w=\sum_i \w_{i}$.

\begin{lemma}
\label{prop:fiber_p_reductive}
Let $\cC_\varphi$ be a quasi fiber--type curve, $\mathrm{char}(\kk)=p$, and $\w := \sum_{i,j} m_{i,j}\alpha_i \sigma_{i,j}\in A^1_{W_\cC,\kk}$
a non-coordinate 1-cochain associated with $\cC_\varphi$ such that $\cC_i$ is completely $p$-reductive w.r.t. $\w_{i}$ for any $i$.
Then 
$\ker(A^1_{W_{\cC}, \kk}\rightmap{\cdot\curlywedge\w} A^2_{W_{\cC},\kk})\subset\cR_{\varphi,\kk}$.
\end{lemma}

\begin{proof}
Let  $\eta \in A^1_{W_\cC,\kk}$ be an arbitrary 1-cochain. Since $m_{i,j}\in\kk^*$ for all $i$ and $j$, we can assume without loss of generality 
$\eta = \sum_{i,j} m_{i,j} \beta_{i,j}\sigma_{i,j}$. Let $i$ be arbitrary and fixed. 
Assume $\eta \curlywedge \w=0$. Following the ideas in the proof of Proposition~\ref{prop:completelypreductive} one obtains:
$$
\left|
\array{cc}
m_{i,j} \alpha_{i}& m_{i,j} \beta_{i,j}\\
m_{i,k} \alpha_{i}& m_{i,k} \beta_{i,k}
\endarray
\right|= 
\alpha_im_{i,j}m_{i,k}
\left|
\array{cc}
1& \beta_{i,j}\\
1& \beta_{i,k}
\endarray
\right| =
0,
$$
for any $j$ and $k$. Since $\alpha_im_{i,j}\in\kk^*$, this equality implies $\beta_{i,j} = \beta_{i,k}$, i.e., $\beta_{i,*}$ does not depend on the choice of $*$.
\end{proof}

\begin{theorem}
\label{thm:p.index}
Let $\cC_\varphi=\cC_1\cup\dots\cup\cC_r$ be a quasi fiber--type curve, $\mathrm{char}(\kk)=p$, such that $m_i=0\in\kk$ for all $i=1,\ldots,r$. 
Define a non-coordinate 1-form $\w:=\sum_{i,j} m_{i,j}\alpha_i\sigma_{i,j}$ associated with $\cC_\varphi$ such that $\sum_i\alpha_i=0\in\kk$. If each $\cC_i$ is completely $p$-reductive w.r.t. $\w_i$ for any $i$, then
$$h^1(A^*_{W_\cC,\kk},\cdot\curlywedge\w)=
\begin{cases}
r-1 & \textrm{ if } \mu_\delta=0\in\kk\ \forall\delta\in\Delta_P \textrm{ local branch and }\ \forall P\in\Sing(\cC) \textrm{ base point,}\\
r-2 & \textrm{ otherwise.}
\end{cases}
$$
Moreover, if $\cC_\varphi$ has index $p>0$, then 
$\ker(A^1_{W_\cC,\kk}\xrightarrow{\cdot\curlywedge\w_\varphi} A^2_{W_\cC,\kk})=\cR_{\varphi,\kk}$.
In particular,
$$h^1(A^*_{W_\cC,\kk},\cdot\curlywedge\w_\varphi)= r-1.$$
\end{theorem}

\begin{proof}
Assume $\mu_\delta=0\in\kk$ at every $\delta\in\Delta_P$ local branch at the base points $P$ of $\cC_\varphi$. Denote $W=W_{\cC
}$ the weak combinatorics of the curve $\cC$. 
The inclusion $\ker(A^1_W \xrightarrow{\cdot\curlywedge\w} A^2_W)\subset \cR_{\varphi,\kk}$ is a consequence of Lemma~\ref{prop:fiber_p_reductive}. The converse inclusion follows since the equations ~\eqref{eq:non_base_P}, \eqref{eq:q_fibr_mod_p} and \eqref{eq:non_base_P_lambda} are satisfied for all $\w'= \sum_{i,j} m_{i,j}\beta_i\sigma_{i,j}\in\cR_{\varphi,\kk}$. Otherwise, 
$$\ker(A^1_W \xrightarrow{\cdot\curlywedge\w} A^2_W)= \bigg\{\w'=\sum_{i,j} m_{i,j}\beta_i\sigma_{i,j}\mid \sum_i\beta_i=0\in\kk \bigg\}\subset \cR_{\varphi,\kk }.$$
For the \emph{moreover} part, the inclusion $\ker(A^1_W \xrightarrow{\cdot\curlywedge\w_\varphi} A^2_W)\subset \cR_{\varphi,\kk}$ is also a consequence of Lemma~\ref{prop:fiber_p_reductive}.
The converse follows from the 
second part of Theorem~\ref{thm:fibered}.
\end{proof}

\begin{corollary}
\label{cor:C_not_p_reductive}
Let $\cC_\varphi= \cC_1 \cup \dots \cup \cC_r$ ($r\geq 2$) be a quasi fiber--type curve of prime index $p$, $\mathrm{char}(\kk)=p$, and $m_i=0\in\kk$ for all $i=1,\ldots,r$. If $\cC_i$ is completely $p$-reductive w.r.t. $\w_{\varphi,i}$ for any $i$ and $m_{i,j} \in \kk^*$ for all $i$ and $j$, then $\cC_\varphi$ is not completely $p$-reductive w.r.t.~$\w_\varphi$.
\end{corollary}
\begin{proof}
Assume that $\cC_\varphi$ is completely $p$-reductive w.r.t. $w_\varphi$. Then, by Proposition 
 \ref{prop:completelypreductive}, $$H^1(A^*_{W_\cC,\ZZ_p},\cdot\curlywedge\w_{\varphi}) = 0.$$ 
 This contradicts Theorem~\ref{thm:p.index} since $r \geq 2$ and hence $\cC$ cannot be completely $p$-reductive. 
\end{proof}

\subsubsection{Examples}
\label{subsec:ex_resonance}

\begin{example}
\label{exam:conics}
The simplest example of Theorem~\ref{thm:fibered} is  a quartic $\cC=\cC_1\cup\cC_2$ given as
a union of two smooth conics with maximal contact at the point $P= [0:0:1]$. For instance $F_1=x^2-yz$, $F_2=y^2-x^2+yz$. Note that
$\cC$ is a fiber--type curve since $F_1+F_2=y^2$. Note that $P$ is a $p$-transversal point and hence $\cC$ is completely $p$-reductive 
w.r.t. any non-zero $\w\in A^1_{W_\cC}$
unless $p=2$. By Proposition~\ref{prop:completelypreductive} and Theorem \ref{thm:inequality}, the only possible roots of twisted Alexander polynomials of order powers of primes are roots of unity of order powers of 2.
Consider hence the case $\kk=\ZZ_2$, and denote by $W=W_\cC$ the weak combinatorics of the curve $\cC$. Note that $A_W^1=\langle \sigma_1,\sigma_2\rangle$,
$A_W^2=\langle \psi_P^{1,2},\bar \psi^1_\infty,\bar\psi^2_\infty\rangle$, where $P=[0:0:1]\in\PP^2$,
and $\sigma_1\wedge\sigma_2=4\psi_P^{1,2}+2\bar\psi^1_\infty-2\bar\psi^2_\infty=0 \, ({\rm mod} \, 2)$. Thus
$\omega=\alpha_1\sigma_1+\alpha_2\sigma_2\in \cR_1(\cC,\kk)^*$ if and only if $\mathrm{char}(\kk)=2$.

On the other hand, there is a surjection
$\pi_1(\tilde X)\surj \pi_1^{\orb} (\PP^1_{(0,0,2)})\cong\ZZ *\ZZ_2\cong \frac{\ZZ x_1*\ZZ x_2 *\ZZ y}{\langle x_1x_2y,y^2\rangle}$
such that $\sigma_i\mapsto x_i$, induced by a surjective map $f: \tilde X \rightarrow \PP^1_{(0,0,2)}$. The curve $\cC$ satisfies $V_1(\cC)=\{(t_1,t_2)\in (\CC^*)^2\mid t_1t_2+1=0\}$, which is the characteristic varieties of the affine complement of $\cC$ with respect to a transversal line $\cC_0$. Hence $f^*(-1,1,-1)=(-1,1)\in V_1(\cC)$.

This can be generalized to any degree where $\cC=\cC_1\cup\cC_2$ is given by $F_1=x^d-yz^{d-1}$, $F_2=y^d-x^d+yz^{d-1}$ satisfying $F_1+F_2=y^d$. Analogously as above, $\cC$ is completely $p$-reductive 
w.r.t. any non-zero $\w\in A^1_{W_\cC}$
if and only if $p\nmid d$. Hence, by Proposition~\ref{prop:completelypreductive} we restrict ourselves to the case $\kk=\ZZ_p$ where $p\mid d$. In that case, $\sigma_1\wedge\sigma_2=d^2\psi_P^{1,2}+d\bar\psi^1_\infty-d\bar\psi^2_\infty=0 \, ({\rm mod}\,  p)$ and hence $\omega=\alpha_1\sigma_1+\alpha_2\sigma_2\in \cR_1(\cC,\kk)^*$ if and only if $\mathrm{char}(\kk)=p|d$.
\end{example}

\begin{example}
\label{ex:TCQuartic1}
Denote by $\cC_4$ the tricuspidal quartic, by $\cC_1$ its bitangent line at the points $\{P,Q\}$, and by $\cC_2$ the conic passing through the three cusps $\{R_1,R_2,R_3\}$ and tangent to $\cC_1$ and $\cC_4$ at $P$. Note that the curve $\cC=\cC_4\cup\cC_2$ is of quasi fiber--type as shown by the existence of the following quasi-toric polynomial identity:
\begin{equation}
\label{eq:quasi-toric}
f_3^2=f_1^2f_4-4f_2^3,
\end{equation}
where $f_1=x+y+z$, 
$f_2=(xy+\xi yz+\xi^2zx)$,
$f_3=x^2y+y^2z+z^2x-x^2z-z^2y-y^2x-3\xi(\xi-1)xyz$,
$f_4=x^2y^2+y^2z^2+z^2x^2-2xyz(x+y+z)$, 
such that $\cC_i=\{f_i=0\}$, and $\xi^2+\xi+1=0$.
Any singular point is $p$-transversal for $\cC$ unless $p=2$, hence the curve $\cC$ is completely $p$-reductive w.r.t. any non-zero $\w\in A^1_{W_\cC}$ except for $p=2$. Hence, by Proposition~\ref{prop:completelypreductive}, the only possible non-trivial cohomology jumping loci in positive characteristic occurs for $p=2$.
Assume $\kk=\ZZ_2$ and consider $X$ the affine complement of $\cC$ and use the index order to mark a preferred branch at each singular point in $S=\{P,R_1,R_2,R_3\}$.  Denote by $W=W_\cC$ the weak combinatorics of the curve $\cC$.
Note that
$A^0_W=\kk$, $A^1_W=\langle \sigma_2,\sigma_4\rangle_\kk$, and
$A^2_W=\langle\bar\psi_P^4,\bar\psi_{R_i}^4,
\bar\psi_{\infty}^2,\bar\psi_{\infty}^4\rangle_\kk.$
It is easy to check that $\sigma_2\wedge\sigma_4=0$.
According to Theorem~\ref{thm:fibered}, and the relation \eqref{eq:quasi-toric}, one can consider
$$
\omega = \alpha\sigma_2+3\beta\sigma_4=\alpha\sigma_2+\beta\sigma_4.
$$ 
This implies that $\omega\in \cR_1(W,\kk)$.
\end{example}

\begin{example}
\label{ex:TCQuartic2}
Consider the Example~\ref{ex:TCQuartic1} where now $\cC$ is the union of the tricuspidal
quartic $\cC_4$, its bitangent $\cC_1$, and $\cC_3=\{f_3=0\}$ the nodal cubic passing tangent to $\cC_4$
at the cusps $\{R_1,R_2,R_3\}$, having a node at $P$ and one of its branches being tangent to both the line $\cC_1$ and the quartic~$\cC_4$. Any cusp is $p$-transversal for $\cC_3$ and $\cC_4$ unless $p=3$. After a $p$-reduction ($p\neq 3$) at $R_i$, the singular point $P$ becomes $p$-transversal unless $p=2,3$. Hence, $\cC$ is completely $p$-reductive w.r.t. any non-coordinate
$\w\in A^1_{W_\cC}$ for $p\neq 2,3$. Consider $\cC_0$ a transversal line to $\cC=\cC_1\cup\cC_4\cup\cC_3$ and use this order
to mark a preferred branch at each singular point in $S=\{P,Q,R_1,R_2,R_3\}$. Denote by $W=W_\cC$ the weak combinatorics of the curve $\cC$. Note that
$A^0_W=\kk$, $A^1_W=\langle \sigma_1,\sigma_3,\sigma_4\rangle_\kk$,
$$A^2_W=\langle\bar\psi_P^4,\bar\psi_P^{\delta_3},\bar\psi_P^{\delta'_3},
\bar\psi_Q^4,\bar\psi_{R_1}^3,\bar\psi_{R_2}^3,\bar\psi_{R_3}^3,
\lambda_2\bar\psi_{\infty}^4,\lambda_3\bar\psi_{\infty}^3, \bar\psi_{\infty}^1\rangle_\kk,$$
where $\delta_3$ (resp. $\delta'_3$) is the transversal branch (resp. tangent branch) of $\cC_3$ at $P$ and
$$\lambda_q:=\begin{cases} 1 & \textrm{ if } \textrm{char}(\kk)=q\\ 0 & \textrm{ otherwise.}\end{cases}$$
Using~\eqref{eq:product2}, one can check that
$$
\array{rcl}
\sigma_1\wedge\sigma_4&=&(2,0,0,2,0,0,0,-\lambda_2,0)=(2,0,0,2,0,0,0,\lambda_2,0),\\
\sigma_1\wedge\sigma_3&=&(0,1,2,0,0,0,0,0,-\lambda_3)=(0,1,2,0,0,0,0,0,2\lambda_3),\\
\sigma_4\wedge\sigma_3&=&(-3,1,2,0,3,3,3,3\lambda_2,-4\lambda_3)=(-3,1,2,0,3,3,3,\lambda_2,2\lambda_3).
\endarray
$$
As above, according with Theorem~\ref{thm:fibered} and the relation $f_1^2f_4-4f_2^3=f_3^2$, we consider
$\omega=2\beta_1\sigma_1+\beta_4\sigma_4+2\beta_3\sigma_3\in A^1_W$ with $\beta_1=\beta_4$. The matrix describing $\cdot\wedge\omega$,
is given by
$$
\left(\begin{array}{rrr}
-2 \, \beta_4 & 2 \, \beta_1 + 3 \, \beta_3 & -3 \, \beta_4 \\
-\beta_3 & -\beta_3 & \beta_1 + \beta_4 \\
-2 \, \beta_3 & -2 \, \beta_3 & 2 \, \beta_1 + 2 \, \beta_4 \\
-2 \, \beta_4 & 2 \, \beta_1 & 0 \\
0 & -3 \, \beta_3 & 3 \, \beta_4 \\
0 & -3 \, \beta_3 & 3 \, \beta_4 \\
0 & -3 \, \beta_3 & 3 \, \beta_4 \\
\beta_4 \lambda_{2} & \beta_1 \lambda_{2} + \beta_3 \lambda_{2} & \beta_4 \lambda_{2} \\
\beta_3 \lambda_{3} & \beta_4 \lambda_{3} & 2\beta_1 \lambda_{3} + 2 \, \beta_4 \lambda_{3}
\end{array}\right)
\simmaps{}{\beta_1=\beta_4}
\left(\begin{array}{rrr}
-\beta_3 & -\beta_3 & 2\beta_1 \\
-2\beta_3 & -2\beta_3 & 4\beta_1 \\
0 & -3\beta_3 & 3\beta_1 \\
\beta_1\lambda_{2} & (\beta_1+\beta_3)\lambda_2 & \beta_1\lambda_{2} \\
\beta_3 \lambda_{3} & \beta_3 \lambda_{3} & 4\beta_1\lambda_{3}
\end{array}\right).
$$
Hence, there is non-trivial cohomology jumping loci only in characteristic 3. In fact, 
$\omega=2\sigma_1+\sigma_4+2\sigma_3\in\cR_1(W,\kk)$ if and only if $\textrm{char}(\kk)=3$.

\end{example}

\subsubsection{Quasi fiber--type curves and twisted Alexander polynomials}

 Recall that, if $\cC=\cC_1\cup\cC_2\cup \ldots \cup\cC_{r}$ is a quasi fiber-type curve, $\cC_i:\{F_i=0\}, \; F_i=\Pi_{j}F_{i,j}^{m_{i,j}}, m_{i,j}>0$ denotes the decomposition of the quasi fiber $\cC_i$ into irreducible components.\\

In this context, \cite[Theorems 3.2 - 3.4]{DPS} can be generalized as follows. 
\begin{thm}
\label{thm:roots}
Assume $\cC=\cC_1\cup\ldots\cup\cC_r$ is a quasi fiber-type curve of index $k$ following the notation above. Then, the primitive roots of order $N$ are zeros of the $(\nu_{i,j})$-twisted Alexander polynomial of multiplicity $b_N$ for $\gcd(\nu_{i,j})=1$, satisfying:
\begin{enumerate}
\item\label{mainthm:cond1}
$\frac{\nu_{i,j}}{m_{i,j}}=\frac{p_i}{q_i}$, $(p_i,q_i)=1$ independent of $j$, 
\item\label{mainthm:cond2}
$N\mid\gcd\left(
\frac{q p_i}{q_i}m_i, i=0,1,...,r\right)$, where $q:={\rm lcm}(q_i)$, $m_0:=k$, and 
$\frac{p_0}{q_0}:=\sum_{i=1}^r\frac{p_i}{q_i}$,
\item\label{mainthm:cond3} 
$b_N\geq\ell_N+r_0-2$, where 
$\ell_N=\#\{i=0,1,...,r\mid q\frac{p_i}{q_i}\neq 0 \mod N, m_i\neq 0\}$ and $r_0:=\#\{i=0,1,...,r\mid m_i=0\}$. 
\end{enumerate}
\end{thm}
\begin{proof}
We present an alternative proof of~\cite[Theorems 3.2 -- 3.4]{DPS} that also applies in this general context.
As discussed in~\eqref{eq:proper-surj}, there exists a morphism 
\begin{equation}
\label{eq:proper-surj2}
f_*:\pi_1(\tilde X)\surj \pi^{\orb}_1(\PP^1_{\bar{m}'})=
\langle \gamma_0,\dots,\gamma_r: \gamma_i^{m_i}= \gamma_0\gamma_1\cdots\gamma_r=1\rangle=
\frac{\ZZ_{m_1}\gamma_1* \cdots*\ZZ_{m_r}\gamma_r}{(\gamma_1\cdots\gamma_r)^{m_0}}
\end{equation}
where $\bar m'=(m_0=k,m_1,\ldots,m_r)$. 
Moreover, if $\mu_{i,j}$ denotes a meridian around the irreducible component $\cC_{i,j}$, then $f_*(\mu_{i,j})=\gamma_{i}^{m_{i,j}}$. Assume there exists a character $\rho:\pi^{\orb}_1(\PP^1_{\bar{m}'})\to\CC^*$, $\rho(\gamma_i)=\xi^{\bar\nu_i}$, whose associated rank-one local system is such that $H^1(\pi^{\orb}_1(\PP^1_{\bar{m}'});\CC_\rho)=d>1$. This implies that $\xi$ is a root of the $(\bar\nu_i)$-twisted Alexander polynomial of $\pi^{\orb}_1(\PP^1_{\bar{m}'})$. Using its pull-back, one has $H^1(\tilde X;\CC_{\rho f_*})\geq d$, and hence the root of unity $\xi\in\CC^*$ is a root of multiplicity at least $d$ of the $(\nu_{i,j})$-twisted Alexander polynomial of $\tilde X$, where $\nu_{i,j}=\frac{m_{i,j}\bar\nu_i}{\bar\nu}$,  $\bar\nu:=\gcd(m_{i,j}\bar\nu_i)$. Define $r_0:=\#\{i=0,1, \ldots ,r\mid m_i=0\}$ and $\ell(\rho)=\#\{i=0,1, \ldots ,r\mid \xi^{\bar\nu_i}\neq 1\}$. By~\cite{BCM}, $b_N=\ell(\rho)+r_0-2$. Note that $\rho$ defines a character if $\xi^{m_i\bar\nu_i}=\xi^{m_0\sum \bar\nu_i}=1$. Condition~\eqref{mainthm:cond1} follows from $\bar\nu_{i}=\bar\nu\frac{\nu_{i,j}}{m_{i,j}}=q\frac{p_i}{q_i}\in\ZZ_{>0}$, where $(p_i,q_i)=1$, which implies $q_i\mid q$. Condition~\eqref{mainthm:cond2} follows from the fact that the order of $\xi$ is $\gcd(\bar\nu_im_i,k\sum\bar\nu_i)$ and $\bar\nu_{i}=q\frac{p_i}{q_i}$. 
Condition~\eqref{mainthm:cond3} follows from $b_N=\ell(\rho)+r_0-2$ since $\ell(\rho)=\ell_N$ only depends on~$N$ as long as $\rho$ is a primitive $N$-th root of unity. 
\end{proof}

\begin{corollary}\label{cor:N|r}
Let $\cC= \cC_1 \cup \ldots \cup \cC_r$ be a reduced fiber--type curve. Then, for any list of integers $\bar \nu=(\nu_{i,j})$, $\gcd \bar\nu=1$, where $\nu_{i,j}=\nu_i$ independently of $j$, the $\bar\nu$-twisted Alexander polynomial of $\cC$ is a multiple of $(t^\nu-1)^{r-2}$, for $\nu:=\sum_{i=1}^r \nu_i$.
\end{corollary}
\begin{proof}
Take, in Theorem \ref{thm:roots},
$m_0=1, \; m_i=0\; (i\geq 1), \; m_{i,j}=1$ and $\nu_{i,j}=\nu_i$, independent of $j$, for any fixed $i$.
 In this case one has $p_0=\sum_{i=1}^r \nu_i$, $q_0=1$, $p_i= \nu_i, \; q_i=q=1$ ($i\geq 1$), $\frac{p_0}{q_0}m_0=\sum_{i=1}^r \nu_i$, $N\mid \sum_{i=1}^r \nu_i$, $r_0=r$, $\ell_N=0$, and $b_N\geq r-2$. 
\end{proof}

\begin{remark}
\label{rem:generalizations}
In particular, the classical case ($\nu_{i,j}=1$) described in \cite[Corollary 3.2]{DPS} can be recovered from Corollary \ref{cor:N|r}.

One can recover \cite[Theorem 3.3]{DPS} from Theorem \ref{thm:roots} for $m_0=k>1$, $m_i=0$ ($i\geq 1$), $m_{i,j}=1$, and $\nu_{i,j}=1$. In this case one has $p_0=r$, $q_0=1$, $p_i=q_i=q=1$ ($i\geq 1$), $N=\frac{p_0}{q_0}m_0=kr$, $r_0=r$, $\ell_N=1$, and $b_N\geq r-1$. 

Finally, one recovers \cite[Theorem 3.4]{DPS} from Theorem \ref{thm:roots} when $m_0=1$, $m_i=0$ ($i\geq 1$), $m_{1,j}=m$, $m_{i,j}=1$ ($i>1$), and $\nu_{i,j}=1$. In this case, one has $p_i=1$ ($i\geq 1$), $q_1=q=m$, $q_i=1$ ($i>1$), $p_0=m(r-1)+1$, $q_0=m$, $N=\frac{p_0}{q_0}m_0=m(r-1)+1$, $r_0=r-1$, $\ell_N=1$, and $b_N\geq r-2$. 
\end{remark}

\begin{remark}\label{rem:thm8.3[PS3]}
Since a reduced $r$-multinet structure on a projective line arrangement $\mathcal{A}$ implies the existence of a reduced pencil of curves of degree $d=|\mathcal{A}|/r$ such that the arrangement is the union of $r$ completely reducible fibers of the pencil,  Corollary \ref{cor:N|r} implies \cite[Theorem 8.3]{PS3}.
\end{remark}

 The above result allows us to prove the existence of roots of twisted Alexander polynomials 
which are not necessarily of prime order, as we can see in the next example. 
\begin{example}
\label{ex:revisited_TCQuartic1}
Let us apply Theorem \ref{thm:roots} to the fiber-type curve from Example \ref{ex:TCQuartic2}, $\cC = \cC_1 \cup \cC_4 \cup \cC_3$ with $\cC_1: f_1^2=0, \; \cC_4: f_4=0, \; \cC_3: f_3^2=0$. We see $\cC$ as the union of two fibers 
$\bar \cC_1:= \cC_1 \cup \cC_4, \bar \cC_2:= \cC_3$ in an index $3$ pencil of sextics. \\
For $m_0=3$, $m_1=m_2=0$, $m_{1,1}=2$, $m_{1,2}=1$, $m_{2,1} = 2$, $m_{i,j} = \nu_{i,j}$, we get $p_1=q_1=p_2=q_2=1$, $p_0=2$, $q_0=1$, $q=1$, and $r_0=2$. For $N=3,6$, one obtains $\ell_N=1$. It follows $b_N \geq 1$, which means that any primitive root of unity of order $6$ or $3$ are roots of the $\varphi$-twisted Alexander polynomial, $\varphi=(m_{i,j})=(2,1,2)$.
Moreover, since $\bar \cC_1$ (resp. $\bar \cC_2$) is completely $3$-reductive w.r.t. $\w_1$ (resp. $\w_2$), by Theorem~\ref{thm:p.index} 
$h^1(A^*_{W_\cC,\ZZ_3},\cdot\wedge\w_\varphi)= 1$, where $\w_{\varphi}=2\sigma_{1,1}+\sigma_{1,2}+2\sigma_{2,1}$. Then, by Theorem~\ref{thm:inequality}, 
$b_3 \leq 1$. It follows $b_3 =1$, i.e. the multiplicity of any primitive root of unity of order $3$ has multiplicity precisely $1$, as a root of the $\varphi$-twisted Alexander polynomial.
However, for $N=2$, one obtains $\ell_2=0$ and hence $b_2\geq0$ does not give any condition. Finally, since $h^1(A_{W_{\cC},\ZZ_2}^*,\cdot\wedge\w_{\varphi})=0$ as shown in Example~\ref{ex:TCQuartic2}, we have in fact~$b_2=0$.
\end{example}

\section{Applications}
\label{sec:applications}

To state the next results consider $\cC_\varphi$ a fiber--type curve obtained as the union of $r$ fibers $\cC_1, \ldots, \cC_r$ in a pencil of degree $d$ curves with $\cC_i= \{ G_i=0\}$.
Let $G_i=\prod_j F_{i,j}^{m_{i,j}}$ with $m_{i,j}>0$ be the decomposition into irreducible components. 
Denote by $\bar m$ the set of multiplicities $(m_{i,j})_{i,j}$.
Moreover, if $m_{i,j}=1$ for all $i,j$, then we call $\cC$ {\it reduced} fiber--type curve.
Denote by $\sigma_{i,j}$ a meridian around the component $\cC_{i,j}= \{F_{i,j}=0\}$. 
If $P \in \Sing(\cC)$ is a base point, then by \eqref{eq:mu_delta} for any branch $\delta$ of $\cC_i$ at $P$ one has that
$$\mu_P(\delta,\cC_j)=\sum_k m_{j,k}\mu_P(\delta,\cC_{j,k})=:\mu_\delta$$ is independent of $j\neq i$. If furthermore $\cC$ is a pencil-type arrangement associated with a Halphen pencil of index $k$ (see Remark \ref{rem:halphen_pencil}), then $k \mid \mu_\delta$.

Let us illustrate how Theorems \ref{thm:inequality}, \ref{thm:p.index} and \ref{thm:roots} can be used in concurrence to compute the multiplicities of certain roots of twisted Alexander polynomials associated to particular fiber--type curves. 

\begin{theorem}
\label{thm:Milnor_reduced}
Let $\cC= \cC_1 \cup \ldots \cup \cC_r$ be a fiber--type curve and $p$ a prime factor of $r$. Assume 
\begin{itemize}
    \item 
$\bar m=(m_{i,j})$ is such that $\gcd(\bar m)=1$ and $\gcd(p,m_{i,j})=1$ for all $i, j$, 
    \item 
$\gcd(p,\mu_{\delta})=1$ for some branch $\delta$ through a base point, and 
    \item 
$\cC_i$ is completely $p$-reductive w.r.t. $\w_i=\sum_{j} m_{i,j}\s_{i,j}$ for every $i$.
\end{itemize}
Then $b_p=r-2$ is the multiplicity of the cyclotomic polinomial $\phi_p$ for the $\bar m$-twisted Alexander polynomial $\Delta_{(\cC,\bar m)}$
of $\cC$. 
\end{theorem}
\begin{proof}
Denote by $b_p$ the multiplicity of a root of the Alexander polynomial $\Delta_{(\cC,\bar m)}$ which is a primitive root of unity of order $p$. 
Theorem \ref{thm:p.index}, via Theorem \ref{thm:inequality}, shows us that $b_p \leq r-2$.
For the inequality $b_p\geq r-2$, one can apply Theorem~\ref{thm:roots} for $\nu_{i,j}=m_{i,j}$ as follows. Note that $\cC$ is a special quasi fiber--type curve where $m_0=1$, $m_i=0$ for $i>0$, and hence $\frac{p_0}{q_0}=r$. In that case, for any $N \mid \frac{p_0}{q_0}m_0=r$, in particular for any $p$ prime factor of $r$, $b_p\geq r-2$, which ends the proof.
\end{proof}

We also have a version of the previous result in the special case when the curve contains a multiple member of a pencil. 

\begin{theorem}
\label{thm:Milnor_nonreduced}
Let $\cC= \cC_1 \cup \ldots \cup \cC_{r+1}$ be a fiber--type curve such that $\cC_1$ is a multiple fiber of multiplicity $k$ and $p$ a prime factor of $1+kr$. Assume 
\begin{itemize}
    \item 
$\bar \nu=(\nu_{i,j})$ is defined as $\nu_{i,j}:=\begin{cases}\frac{m_{1,j}}{k} & \textrm{ for } i=1\\ m_{i,j} & \textrm{ otherwise}\end{cases}$ such that $\gcd(\bar \nu)=1$ and $\gcd(p,\nu_{i,j})=1$, for all $i, j$, 
    \item 
$\gcd(p,\mu_{\delta})=1$ for some branch $\delta$ through a base point, and 
    \item 
$\cC_i$ is completely $p$-reductive w.r.t. $\w_i=\sum_{j} \nu_{i,j}\s_{i,j}$ for every $i$.
\end{itemize}
Then $b_p=r-1$ is the multiplicity of the cyclotomic polinomial $\phi_p$ for the $\bar \nu$-twisted Alexander polynomial $\Delta_{(\cC,\bar \nu)}$
of $\cC$. 
\end{theorem}
\begin{proof}
Denote by $b_p$ the multiplicity of a root of the Alexander polynomial $\Delta_{(\cC,\bar \nu)}$ which is a primitive root of unity of order $p$. By Theorem~\ref{thm:inequality}, in order to prove $b_p\leq r-1$ it is enough to consider the 1-form $\w=\sum_{i,j} \nu_{i,j} \sigma_{i,j}$ and show that 
$h^1(A^*_{W_{\cC},\kk},\cdot\wedge\w)\leq r-1$. We will do this using Theorem~\ref{thm:p.index}. Note that $\w$, as defined above, is a non-coordinate 1-form associated with $\cC$, $\alpha_{1,j}=k$ and $\alpha_{i,j}=1$ for $i>1$. The condition $\sum_i\alpha_i=\frac{1}{k}+r=0\in\kk$ follows from the hypothesis $p\mid (1+kr)$. Since $\cC_i$ is completely $p$-reductive w.r.t. $\w_i$, $\gcd(p,m_{i,j})=\gcd(p,\nu_{i,j})=1$, and $\gcd(p,\mu_\delta)$ for some local branch $\delta$ at a base point, Theorem~\ref{thm:p.index} implies 
$h^1(A^*_{W_{\cC},\kk},\cdot\wedge\w)\leq (r+1)-2=r-1$.
For the inequality $b_p\geq r-1$ one can apply Theorem~\ref{thm:roots} for $\nu_{i,j}$. Note that $\cC$ is a special quasi fiber--type curve where $m_{1,j}=k\nu_{1,j}$, $m_{i,j}=\nu_{i,j}$ for $i>1$, $m_0=1$, $m_i=0$ for $i>0$, $p_1=1$, $q_1=k$, $p_i=q_i=1$ for $i>1$, $q=k$, $\frac{p_0}{q_0}=\frac{1}{k}+r$, and $N\mid (1+kr)$. In this case, $\ell_N=0$, $r_0=r+1$, and hence $b_N\geq r-1$. In particular for any $p$ prime factor of $r$, $b_p\geq r-1$, which ends the proof.
\end{proof}

\begin{remark}
\label{rem:ex_milnor}
 
In light of the the above two results, 
let us revisit the Halphen pencil type examples from \S\ref{subsec:milnor}.

Consider Examples \ref{ex:thm6.2_DPS} and \ref{ex:DPS_thm6.4}. In both these examples, for the conic-line arrangements $\mathcal{A}$, we get $b_7 = 2$, as a consequence of Theorem \ref{thm:Milnor_nonreduced}.
 As for the conic-line arrangements $\mathcal{B}$ from \ref{ex:thm6.2_DPS} and \ref{ex:DPS_thm6.4}, the inequalities $b_3, b_9 \leq 1$ are also a consequence of Theorem 
\ref{thm:p.index}. Theorem \ref{thm:Milnor_reduced} proves that $b_3=1$ for the conic-line arrangements $\mathcal{B}$ from the same examples.
Still for the conic-line arrangements $\mathcal{B}$, the inequality $b_2 \leq 2$ follows from Theorem~\ref{thm:p.index}, taking into account the complete $2$-reductiveness of the three pencil members that compose $\mathcal{B}$. 

Consider secondly Example \ref{ex:Thm_5.4_DPS}. In this example, Theorem~\ref{thm:p.index}, via Theorem \ref{thm:inequality}, immediately shows that $b_2,b_4,b_8 \leq 3$.
On the other hand, Theorem \ref{thm:roots} shows that $b_2, b_4 \geq 2$ and $b_8 \geq 3$. To see this, take   $m_0=2$, $m_1=m_2=m_3=m_4 =0$, 
$m_{i,j}=\nu_{i,j}=1$ for all $i,j$. 
 We get $p_i=q_i=1$ for all $i \geq 1$, $p_0=4$, $q_0=1$, $q=1$, and $r_0=4$.
 For $N=2,4$, one obtains $\ell_N=0$ and for $N=8$, $\ell_N=1$.
Hence, $2 \leq b_2, b_4 \leq  3$ and $b_8 = 3$.
\end{remark}

\begin{remark}
\label{rem:another_Yoshinaga_counterex}
Let us take another look at  Example \ref{ex:Thm_5.4_DPS}.
We already know that $b_2=b_4=2$, see \cite[Theorem 5.4]{DPS}. This means that the inequality \eqref{eq:mod_ineq} is strict, since it becomes $2<3$, for $p=2$ and $s \in \{1,2\}$.\\
\end{remark}

\section{A look into the strict inequality: Yoshinaga's example}
\label{sec:Yoshinaga}
In this section we will study an example proposed by Yoshinaga in~\cite{Y} as a counterexample to the conjecture that the modular inequality \eqref{eq:mod_ineq} is in fact an equality for arrangements of hyperplanes when $k=2,3,4$ -- see~\cite[Conjecture 1.9]{PS3}. This counterexample is given by the icosidodecahedron arrangement $\cA_{\ID}$. The arrangement $\cA_{\ID}$ contains $16$ lines with $15$ quadruple points, and the rest are double points. Using the classical notation introduced by Hirzebruch, we have $t_4=15$, $t_3=0$, and $t_2=30$. One can check that $\cA_{\ID}$ is a quasi fiber--type curve of index $2$. This can be seen as follows: $\cA_{\ID}$ includes exactly 6 lines containing five quadruple points each (and no other singular points). We denote those lines $\cL_{11}, \ldots,\cL_{16}$. The remaining ten lines, $\cL_1, \ldots ,\cL_{10}$, contain three quadruple points each and six double points. There exists a smooth conic $\cC_2=\{f_2=0\}$ and a quintic $\cC_5=\{f_5=0\}$ such that
\begin{equation}
\label{eq:YoshinagaToric}
L_1+L_2f_2^2=f_5^2,
\end{equation}
where $\{L_1=0\}=\cL_1\cup \ldots \cup\cL_{10}$ and $\{L_2=0\}=\cL_{11}\cup \ldots \cup\cL_{16}$. According to Theorem~\ref{thm:fibered}, the resonance variety $\cR_1(W_{\cA_{\ID}},\ZZ_2)$ includes any $\w=\alpha_1\w_1+\alpha_2\w_2$, where $\w_1=\sum_{i=1}^{10} \sigma_i$ and $\w_2=\sum_{i=11}^{16} \sigma_i$. Hence, $\dim_{\ZZ_2} H^1(A^*_{W_{\cA_{\ID}},\ZZ_2},\cdot \curlywedge \w)\geq 1$.
The equality follows easily using the double points and Corollary~\ref{cor:double_points} to show that $\w=\alpha_1\w_1+\sum_{i=11}^{16} \alpha_i\sigma_i$. Using the relations~\eqref{eq:det} at the quadruple points, one obtains that $\w$ is resonant if and only if $\w=\alpha_1\w_1+\alpha_2\w_2$, and thus 
\begin{equation}
\label{eq:dim_aomoto}
    \dim_{\ZZ_2} H^1(A^*_{W_{\cA_{\ID}},\ZZ_2},\cdot \curlywedge \w)=1.
\end{equation}

This gives an alternative proof of~\cite[Proposition 4.2]{Y}. 

Since $|\cA_{\ID}|=16$, the only non-trivial possible roots of the classical Alexander polynomial are powers of $2$, so $N = 2,4,8,16$. Denote by $b_N$ the multiplicity of the primitive root of order $N$. By Theorem \ref{thm:inequality} and \eqref{eq:dim_aomoto}, 
$$b_N \leq 1,\; \forall N.$$

Also, in~\cite[Proposition 4.2]{Y}, Yoshinaga shows that $-1$ is not a root of the Alexander polynomial.
In light of~\eqref{eq:YoshinagaToric}, note that $\cC=\{L_1L_2=0\}$ can be considered as a reduced quasi fiber-type curve where $m_0=2$, $m_1=1$, $m_2=2$, and $m_{i,j}=1$. Since we are interested in the classical Alexander polynomial, $\nu_{i,j}=1$ and hence, by Theorem~\ref{thm:roots}, one has $q_0=p_1=p_2=q_1=q_2=1$, $p_0=2$, $N=2$, $r_0=1$, and $\ell_2=1$. This implies that $0=\ell_2 + r_0 -2\leq b_2\leq 1$, which becomes in fact an equality~$b_2=0$ as shown in~\cite{Y}.

\end{document}